\documentclass[10pt]{amsart}
\usepackage{graphicx}
\usepackage{amsmath}
\usepackage{graphicx}
\usepackage{amsfonts}
\usepackage{amssymb}
\usepackage{amsmath, amssymb ,amsthm, amsfonts, amsgen}
\DeclareGraphicsExtensions{.eps,.bmp,.jpg,.pdf,.mps,.png,.gif}
\numberwithin{equation}{section} 
  \newtheorem{Theorem}{Theorem}[section]
\newtheorem{Lemma}[Theorem]{Lemma}
\newtheorem{Proposition}[Theorem]{Proposition}

\newtheorem{Remark}[Theorem]{Remark}

\usepackage[francais,english]{babel}
\usepackage[T1]{fontenc}
\usepackage{times}
\newcommand\var{\varepsilon}
\newcommand\display {\displaystyle}
\usepackage {floatrow}
\usepackage {caption}
\DeclareGraphicsExtensions{.eps,.bmp,.jpg,.pdf,.mps,.png,.gif}
\def\var{\varepsilon}
\def\tscale{\rightharpoonup\rightharpoonup}
\def\weak{\rightharpoonup}

\usepackage{amsfonts}
\def\var{\varepsilon}
\def\display{\displaystyle}
\def\weak{\rightharpoonup}

\def\tscale{\overset{2-sc}{\weak}}
\def\scvge{\longrightarrow}

\begin{document}
\thispagestyle{empty}
\title[On the limit spectrum of a degenerate operator]{On the limit spectrum of a degenerate operator in the framework of periodic homogenization or singular perturbation problems}

\author{Ka\"{\i}s Ammari}
\address{LR Analysis and Control of PDEs, LR 22ES03, Department of Mathematics, Faculty of Sciences of Monastir, University of Monastir, Tunisia}
\email{kais.ammari@fsm.rnu.tn}

\author{Ali Sili}
\address{Institut de Math\'ematiques de Marseille (I2M), 
UMR 7373, Aix-Marseille Universit\'e,  CNRS, CMI,  39 rue 
F. Joliot-Curie, 13453 Marseille cedex 13, France} 
\email{ali.sili@univ-amu.fr}

\begin{abstract} In this paper we perform the analysis of the spectrum of a degenerate operator $A_\var$ corresponding to the stationary heat equation in a $\var$-periodic composite medium having two components with high contrast diffusivity.  We prove that although $ A_\var$ is a bounded self-adjoint operator with compact resolvent, the limits of its eigenvalues  when the size $\var$ of the medium tends to zero, make up a part of the spectrum of a unbounded operator $ A_0$, namely the eigenvalues of $ A_0$ located on the left of the first eigenvalue  of the bi-dimensional Laplacian with homogeneous Dirichlet condition on the boundary of the representative cell. We also show that the homogenized problem does not differ in any way from the one-dimensional problem obtained in the study of the local reduction of dimension induced by the homogenization.
\end{abstract}

 
\maketitle

\tableofcontents

\section{Introduction, setting of the problem and statement of the results}
The purpose of the present work is the asymptotic analysis of the eigenelements of a spectral problem in the framework of the homogenization of a periodic composite medium made up of a $\var$-periodic set of parallel vertical fibers $F_\var$ surrounded by a matrix $M_\var$ having better properties; more precisely, we consider the following problem 
\begin{equation} \label{newstrongformulationH} \left\{\begin{array}{ll}   
    \display  A_\var u_\var = \lambda_\var u_\var  \quad \hbox{in} \,\, \,  \Omega, \\\\
    \hbox{where} \,\, \,  
 \display  A_\var u =  \display    - \var^2 \Delta u   \chi_{F_\var} - \Delta u \chi_{M_\var} \quad \forall \, u \in D({ A}_\var), \\\\
 \hbox {with} \, \, \, \display  D({ A}_\var)  = \left\{ u \in V_h, \, \, \,   A_\var u \in L^2(\Omega), \, \, \, 
    \frac{\partial u}{\partial n} \chi_{\partial F_\var} = - \frac{1}{\var^2} \frac{\partial u}{\partial n} \chi_{\partial M_\var}\right\}, \end{array} \right.\end{equation}
with the following notations: 

\medskip

$\Delta$ denotes the classical Laplacian operator,  $ \Omega$ denotes a bounded rectangular open set of $\mathbb R^3$ of the form $ \Omega : =  \omega \times (0, L)$, $ \omega $ being a domain of $\mathbb R^2$ and $L$ is a positive number, $ \display  \frac{\partial u}{\partial n} \chi_{\partial M_\var}$ (resp. $ \display  \frac{\partial u}{\partial n} \chi_{\partial F_\var}$) denotes the outer normal to the lateral boundary of $M_\var$ (resp. $F_\var$).  The two horizontal variables $x':=(x_1,x_2)$ or $y:=(y_1,y_2)$ will play a different role from that of the vertical one $x_3$. The gradient and the Laplacian with respect to the horizontal variables will be denoted respectively by $\nabla'$ and $ \Delta'$. \par
The space $V_h$ ($h$ stands for homogenization) is defined by
\begin{equation}\label{l'espace V} V_h:= \big\lbrace u \in H^1(\Omega), \, \, u(x',0) = u(x',L) = 0 \, \hbox{ a.e.} \, x'=(x_1, x_2) \in \omega \big\rbrace, \end{equation}
hence, $V_h$ is the subspace of functions in $H^1(\Omega)$ which vanish on the lower and the upper faces of $\Omega$ which is assumed to be the reference configuration of a composite medium whose two components are a set $F_\var$ of vertical cylindrical fibers and its complement, the matrix $M_\var$.  

\medskip 

 Hence, the projection on the horizontal $x'$-plane of the set $F_\var$ is made up of a $\var$-periodic set of disks while the complement of such set represents the projection of $M_\var$.  The characteristic functions of $F_\var$ (resp. $M_\var$) are denoted by $\chi_{F_\var}$ (resp. $\chi_{M_\var}$). The fibers are distributed in $\Omega$ with a period of size $\var$ and the ratio between the conductivity coefficients of the two components is $\var^{-2}$. Throughout the paper, for a measurable set $B$ we denote by $ \vert B \vert$ its Lebesgue measure and by $\chi_B$ its characteristic function. A generic positive constant the value of which may change from a line to another will be denoted by $K$.  \par
 The geometric configuration of the medium may be described more precisely as follows.

\medskip

 Let  $C$ be a square of $\mathbb R^2$ and let $D$ be a disk strictly contained in $C$. The complement of $\overline D$ in $C$ will be denoted by $M'$ in such a way that $C= M' \cup \overline D$. We then define
 \begin{equation}\label{geoemetryhomogenization} \left\{\begin{array}{ll}   \display C^i_\var = (\var C + \var i) \times (0,L); \, \, \, I_\var = \lbrace i \in \mathbb Z^2, \, \, C_\var^i \subset \Omega \rbrace, \quad \omega = {\displaystyle  \bigcup_{i \in I_\var}} (\var C + \var i) ; \\\\ \Omega = {\display \bigcup_{i \in I_\var} }C^i_\var = \omega \times (0,L), \quad F_\var = {\displaystyle \bigcup_{i \in I_\var}} F^i_\var, \quad  F^i_\var = D_\var^i \times (0,L)= (\var \overline D + \var i ) \times (0,L), \\\\
 M_\var = {\display \bigcup_{i \in I_\var}} M^i_\var , \,\,  \, M^i_\var = ( C^i_\var \setminus \overline D^i_\var) \times (0,L) \, \,  \mbox{in such a way} \, 
 \quad  \Omega = F_\var \bigcup M_\var. \end{array} \right.
\end{equation}
  
  \begin{figure} [!ht]
  \begin{center}
    \includegraphics[width= 0.60 \linewidth]{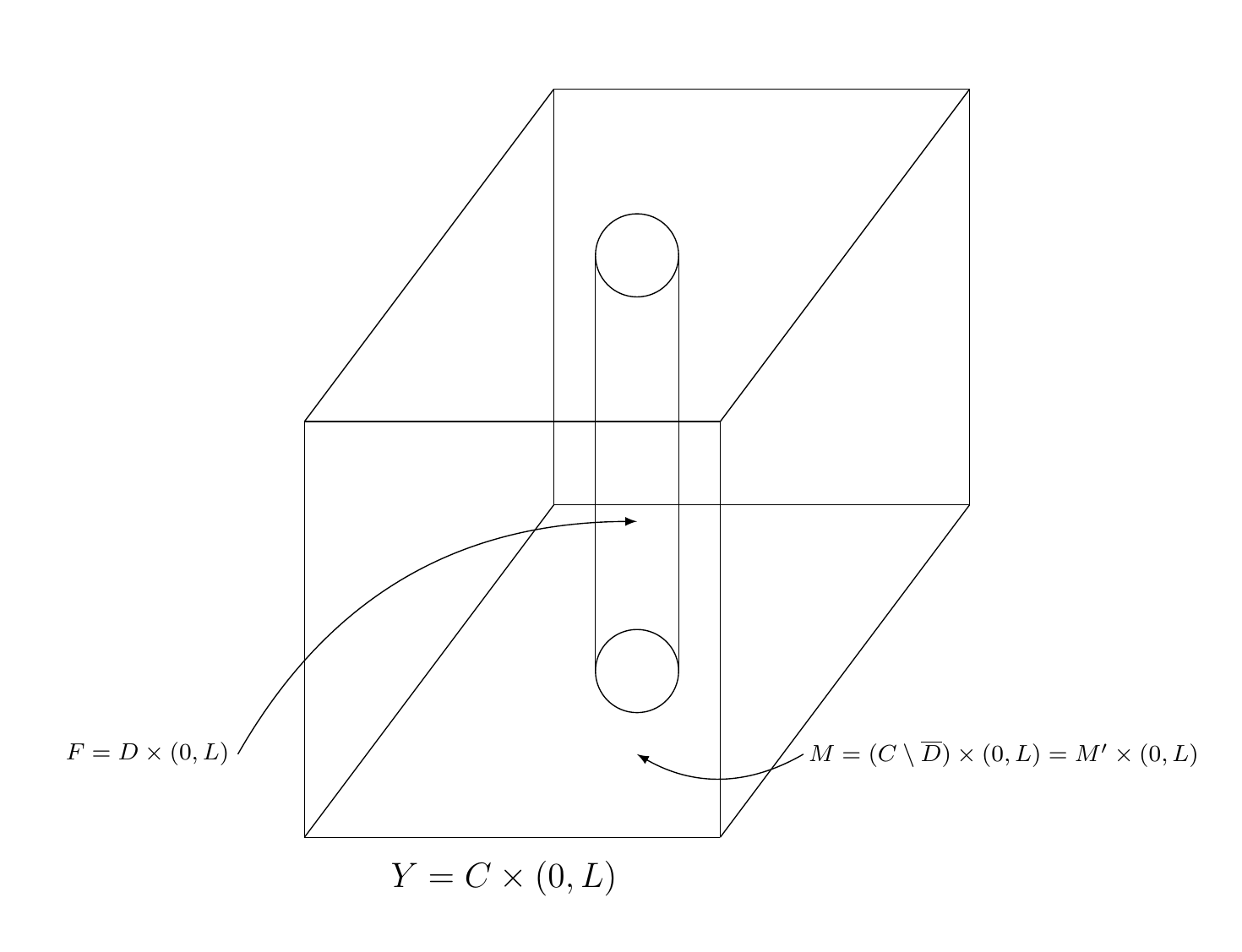}
  
    \end{center}
  \caption{The composite structure after dilation which is also the reference cell $Y$ in the homogenization setting.}
 \label{fig:exmp}
 \end{figure}

  In \textbf{Figure 1} we have represented the representative cell $Y:= C \times (0, L) = (\overline D \cup M') \times (0, L)$ which represents also the composite structure after dilation. 
	
	\medskip
	
  When dealing with the homogenization of a problem posed on the domain $\Omega$ with a geometry given by (\ref{geoemetryhomogenization}), a reduction of dimension $3d-1d$ appears locally in each cell $ \display C^i_\var $; a natural question then arises: what is the relationship between the homogenization process and the associated local reduction of dimension? One of the main points of this article is to show that homogenization does not lead to any new phenomena compared to the local reduction of dimension. 
More precisely, we will show that homogenization is merely a periodic repetition of the local dimension reduction phenomenon. \par
The geometry of the reduction of dimension $3d-1d$ problem is the following: the composite medium consists of a single fiber $F_\var:= (\var  D) \times (0,L)$  surrounded by the matrix $M_\var= \var M' \times (0,L)= \bigl(\var (C\setminus \overline D) \bigr) \times (0,L)$ in such a way  the global domain depends now on the small parameter $\var$; it is defined by $ \Omega_\var:= (\var C) \times (0,L)= F_\var \cup M_\var$ and it may be viewed as the configuration of a thin structure with the characteristic parameter $\var$.  
	
	\medskip
	
   In this setting, the spectral problem (\ref{newstrongformulationH}) takes the following form
  \begin{equation} \label{newstrongformulationS} \left\{\begin{array}{ll}   
    \display  A_\var v_\var = \lambda_\var v_\var  \quad \hbox{in} \,\, \,  \Omega_\var, \\\\
    \hbox{where} \,\, \,  
 \display  A_\var v =  \display    - \var^2 \Delta v   \chi_{F_\var} - \Delta v \chi_{M_\var} \quad \forall \, v \in D({ A}_\var), \\\\
 \hbox {with} \, \, \, \display  D({ A}_\var)  = \left\{ v \in V^\var_s, \, \, A_\var v \in  L^2(\Omega_\var), \, \, \, 
    \frac{\partial v}{\partial n}  \chi_{\partial (\var D)} =  - \frac{1}{\var^2} \frac{\partial v}{\partial n}\chi_{\partial (\var M')} \right\}, \end{array} \right.\end{equation}
  where the space $V^\var_s$ (the subscript "s" stands for singular perturbation) is now defined by
  \begin{equation}\label{l'espace Vbis} V_s^\var:= \left\{v \in H^1(\Omega_\var), \, \, v(x',0) = v(x',L) = 0 \, \hbox{ a.e.} \, x'=(x_1, x_2) \in \var C \right\}. \end{equation}
  
     In order to deal with a problem posed on the fixed domain $\Omega:= C \times (0,L)$, we introduce the classical scaling $u_\var(y',x_3) = v_\var(\var y',x_3), \, \, y' \in C$ which implies 
   \begin{equation} \label{Gradientperturbé}  \display \nabla'_y u_\var(y, x_3) = \var \nabla' v_\var(\var y, x_3) = \var \nabla' v_\var(x', x_3), \, \, \, \forall \, (x', x_3) \in (\var C) \times (0, L),  \end{equation}  (this approach is of course not applicable in the homogenization setting in which we have to deal with $\var^{-2}$ such thin structures). This change of variables transforms the problem (\ref{newstrongformulationS}) into the following singular perturbation problem (recall that $M':= C \setminus \overline D$),
  
  \begin{equation} \label{newstrongformulationS'} \left\{\begin{array}{ll} 
  A_\var u = \lambda_\var u \, \, \, \hbox{in}   \, \, \Omega, \\\\
 \hbox{where} \, \, \,  A_\var u =  \display   \left( - \Delta' u  - \var^2 \frac{\partial^2 u}{\partial x_3^2} \right) \chi_D + \left( - \frac{1}{\var^2}\Delta' u  -  \frac{\partial^2 u}{\partial x_3^2} \right) \chi_{M'},  \quad \forall \, u \in D({ A}_\var), \\\\
 \hbox{with} \\\\
  \display  D({ A}_\var)  = \left\{ \display u \in V_s, \, \, A_\var u \in L^2(\Omega), \, \, \,   \frac{\partial u}{\partial n} \chi_{\partial D} = -\frac{1}{\var^2} \frac{\partial u}{\partial n}\chi_{\partial M'} \right\}, \end{array} \right.\end{equation}
  $V_s$ being the space $V_s^\var$ corresponding to $\var=1$ and defined in (\ref{l'espace Vbis}). 
	
	\medskip
	
  Note that the study of the asymptotic behavior of (\ref{newstrongformulationS}) is the so-called reduction of dimension problem $3d-1d$ since when $\var$ goes to zero the three dimensional domain $\Omega_\var= (\var C) \times (0,L)$ looks like the segment $(0,L)$. \par
   Remarkably, it appears that the homogenized problem is very similar to the limit problem describing the one-dimensional model in the local $3d-1d$ reduction of dimension as explained in \cite{Sili0} (see also \cite{{MurSi},{ParSi}}). This similarity is essentially due to the absence of oscillations in the vertical direction, whereas oscillations in the horizontal plane induce a local reduction of dimension. \par
    We take advantage of that remark to limit ourselves to the complete study of the $3d-1d$ problem which is technically simpler than the homogenization problem and we will only state the results within the framework of homogenization by referring to \cite{Sili0} for an adaptation of the proofs to the homogenization.   
 \par    
Homogenization of a medium with high contrast between its components leads in general to a limit model described by an equation with significant differences compared with the equation of  the media at the scale $\var$, see \cite{ADH}, \cite{BPP},  \cite{CaiDi}, \cite{GauSi}, \cite{CharSi}, \cite{MurSi}, \cite{ParSi}, \cite{Sili}, \cite{Sili2}, \cite{Zhik}. Other settings have been studied in \cite{AllCap}, \cite{Kes}, \cite{KeSab}, \cite{MelNa}. \par
\noindent Of course, this rule also fits for spectral problems, see for instance \cite{Zhik},  \cite{LeNaTa}, \cite{Sili0}. 

\medskip

To describe the behavior of the eigenvalues of (\ref{newstrongformulationS'}) ( $3d-1d$ problem) or  (\ref{newstrongformulationH}) (homogenization) we use the variational formulation.    Note that for a fixed $\var$, $A_\var$ defined either by (\ref{newstrongformulationH}) or by (\ref{newstrongformulationS'}) is a bounded selfadjoint operator with compact resolvent so that one can state the following well known result.
\begin{Proposition} 
  Problem (\ref{newstrongformulationH}) (or problem (\ref{newstrongformulationS'})) admits a sequence of eigenvalues $(\lambda_\var^k)_k, \, \, 0 < \lambda_\var^1 \leq \lambda_\var^2 \leq ... \leq  \lambda_\var^n \leq ...,$ with $ \display  \lim_{k \to \infty} \lambda_\var^k = + \infty$ while the associate eigenvectors $(u_\var^k)_k$ may be chosen as an orthonormal basis of  $L^2(\Omega)$.
 \end{Proposition}
Taking into account this result, the variational formulation of (\ref{newstrongformulationH}) and of (\ref{newstrongformulationS'}) are respectively the following ones
  \begin{equation} \label{FV1} \left\{\begin{array}{ll}  u_\var^k \in V_h, \\\\  \displaystyle \int_\Omega ( \var^2  \nabla u_\var^k \nabla \phi \chi_{F_\var} +   \nabla u_\var^k \nabla \phi  \chi_{M_\var} \bigr) dx   = \lambda_\var^k \display \int_\Omega u_\var^k \phi \ dx, \\\\ \forall \, \phi \in V_h,   
  \end{array} \right. \end{equation}
  \begin{equation} \label{FV2} \left\{\begin{array}{ll}  u_\var^k \in V_s, \\\\  \displaystyle \int_\Omega \left( \left(\nabla' u_\var^k \nabla' \phi + \var^2 \frac{\partial u_\var^k}{\partial x_3} \frac{\partial \phi}{\partial x_3}\right) \chi_{F} +   \left(\frac{1}{\var^2}\nabla' u_\var^k \nabla' \phi +  \frac{\partial u_\var^k}{\partial x_3} \frac{\partial \phi}{\partial x_3}\right) \chi_{M} \right) dy \ dx_3    \\\\ = \lambda_\var^k \display \int_\Omega u_\var^k \phi \ dy \ dx_3, \, \, \,  \forall \, \phi \in V_s,  
  \end{array} \right. \end{equation}
 where $F:= D \times (0,L)$ and $M:= (C\setminus \overline D) \times (0,L)$. \par
 
 We prove in Theorem 1.3 below that for each $k$, the limit $\lambda_k$ of the sequence of eigenvalues $(\lambda_\var^k)_\var$ of (\ref{newstrongformulationS'}) ( $3d-1d$ problem) is either equal to the first eigenvalue $\mu_1$ of the bidimensional Laplacian in the disk $D$ with homogeneous Dirichlet boundary condition or is on the left of $\mu_1$; furthermore, if $\lambda_k$ fulfills  $0< \lambda_k < \mu_1$,  then $\display s(\lambda_k):= \lambda_k \Bigl( 1 + \frac{\vert D \vert }{\vert C \setminus D\vert } + \frac{ \lambda_k}{\vert C \setminus D \vert} \displaystyle \int_{D} u^k_0 \ dy \Bigr)$ is an eigenvalue of $\display -\frac{d^2}{dx^2_3}$ in $(0,L)$ with homogeneous Dirichlet boundary condition; more precisely, $\lambda_k$ is a solution of the following coupled system 
 
 \begin{equation} \label{FF2} \left\{\begin{array}{ll}  u^k_0(y) \in H^1(C),\quad
 \displaystyle - \Delta'_y u^k_0  = \lambda_k u^k_0 + 1\,\, \, \,  \hbox{in} \,\,  D,  \\\\
  u^k_0 = 0 \quad   \hbox{on} \, \, \, \partial D, 
   \\\\
  v_k \in  H^1_0(0, L)),  \quad \display - \frac{d^2 v_k}{dx_3^2} = \lambda_k \Bigl( 1 + \frac{\vert D \vert }{\vert C \setminus D\vert } + \frac{ \lambda_k}{\vert C \setminus D \vert} \displaystyle \int_{D} u^k_0 \ dy \Bigr) v_k \quad \hbox{in} \, \, \,  (0,L).
  \end{array} \right. \end{equation}
 
 Similar results (Theorem 1.5) are obtained for the homogenization problem; the limit $\lambda_k$ of the sequence of eigenvalues $(\lambda_\var^k)_{\var}$ of (\ref{newstrongformulationH}) is either equal to the first eigenvalue $\mu_1$ of the bidimensional Laplacian in the disk $D$ with homogeneous  Dirichlet boundary condition or is on the left of $\mu_1$ and satisfies the coupled system 
 \begin{equation} \label{FF1} \left\{\begin{array}{ll}  u^k_0 \in  H^1_{\#}(C),\quad
 \displaystyle - \Delta'_y u^k_0  = \lambda_k u^k_0 + 1\,\, \, \,  \hbox{in} \,\, D,  \\\\
   u^k_0 = 0 \, \, \,   \hbox{on} \, \, \partial D, 
   \\\\
 v_k \in L^2(\omega; H^1_0(0, L)),  \\\\
  \display - \frac{\partial^2 v_k}{\partial x_3^2} = \lambda_k \Bigl( 1 + \frac{\vert D \vert }{\vert C \setminus D\vert } + \frac{ \lambda_k}{\vert C \setminus D \vert} \displaystyle \int_{D} u^k_0 \ dy \Bigr) v_k \quad \hbox{in} \, \, \,  \Omega,
  \end{array} \right. \end{equation}
	where $H^1_\# (C)$ denotes the space of $C-$periodic functions in $H^1(C)$. 
	
	\medskip
 
 Let us notice the very close analogy between the two limit problems. Roughly speaking, the vibrations at the limit are those of the string $(0,L)$ modeled by the equation on $v$, which nevertheless contains a memory term $ \display \int_D u_0(y) dy$ inherited from the initial vibrations in the horizontal directions of the soft part of the medium (the fibers) giving rise to a nonlocal system. The first equation of  (\ref{FF1}) is exactly the first one in (\ref{FF2}) (the boundary condition $\display  u^k_0 = 0 \, \, \,   \hbox{on} \, \, \partial D$ 
 allows to extend $u_0^k$ by zero so that it may be viewed as an element of $H^1_{\#}(C)$ as in (\ref{FF1})); hence the only difference between (\ref{FF1})  and (\ref{FF2}) lies in  $v_k$ arising in (\ref{FF1}) is a function depending also on the variable $x' \in \omega$ while in (\ref{FF2}) it depends only on the vertical variable $x_3 \in (0,L)$.  The latter simply means that  the homogenized model is a duplication through the horizontal plane $\omega$ of the phenomenon occurring in each cell. \par
 Note the contrast with limit problems obtained with uniformly bounded operators with respect to the small parameter which are generally of the same nature as the original ones, see for instance  \cite{Kes}, \cite{KeSab}, \cite{Van}. 

\medskip

Note also that the existence and the uniqueness of $u^k_0 $  in (\ref{FF2}) is ensured by the fact that $ \lambda_k$ belongs to the resolvent $ \rho( -\Delta'_y)$ of $ \display - \Delta'_y$ since  
$ \lambda_k < \mu_1$. \par
The fact that $ \display \int_{D} u^k_0 \ dy \not= 0$ will be proved in section 2, see (\ref{lowerboundlast}), using the constant $\mu_1^{-1}$  in the Poincar\'e inequality.

  \begin{Remark}
   It is natural to ask what is the relationship between the problem (\ref{FF2}) (or the problem (\ref{FF1})) and the classical formulation of eigenvalue problems. In fact, (\ref{FF2}) is derived from the system (\ref{FF2bis}) which in turn is derived from the equation (\ref{FV2ter}) satisfied by the pair $(u_k,v_k)$, see the details of the proof in section 2 below. If one integrates the first equation of (\ref{FF2bis}) over $D$, we get an equivalent formulation of (\ref{FF2bis}) as follows
   \begin{equation} \label{LimitOperator} \left\{\begin{array}{ll}  u_k(y,x_3) \in L^2((0,L); H^1(C)),\quad
 \displaystyle - \Delta'_y u_k (y,x_3) = \lambda_k u_k \,\, \, \,  \hbox{in} \,\,  D \times (0, L),  \\\\
  u_k = v_k \quad   \hbox{on} \, \, \, \partial D \times (0, L), 
   \\\\
        
   v_k \in  H^1_0(0, L),  \quad \display - \frac{d^2 v_k}{dx_3^2} +\frac{1}{\vert C \setminus D \vert} \displaystyle \int_{\partial D} \frac{\partial u_k}{\partial n}\ d\sigma = \lambda_k v_k  \quad \hbox{in} \, \, \,  (0,L).
  \end{array} \right. \end{equation}
  Another equivalent formulation of (\ref{LimitOperator}) is the following 
  \begin{equation}\label{LimitOperatorbis}
 A_0 \begin{pmatrix} u_k \\ v_k \end{pmatrix} = \lambda_k \begin{pmatrix} u_k \\ v_k \end{pmatrix}
  \end{equation}
  where the operator $A_0$ is defined by $ \display A_0: D(A_0) \to H:=  L^2(\Omega) \times L^2(0, L)$ with
  \begin{equation}\label{LimitOperatorter} \left\{\begin{array}{ll}
  \display D(A_0) = 
  \left\{ \begin{pmatrix} u \\ v \end{pmatrix} \in L^2(0,L; H^1(D) \times H^1_0(0,L); \, A_0 \begin{pmatrix} u \\ v \end{pmatrix} \in H,  \,  u = v \, \, \hbox{on} \, \, \partial D \right\}, \\\\
  
    A_0 \begin{pmatrix} u \\ v \end{pmatrix} = \display \begin{pmatrix} -\Delta'_y u \\ - \display \frac{d^2 v}{dx_3^2} +\frac{1}{\vert C \setminus D \vert} \displaystyle \int_{\partial D} \frac{\partial u}{\partial n}\ d\sigma \end{pmatrix},  \qquad \forall \, \begin{pmatrix} u \\ v \end{pmatrix}  \in D(A_0).
\end{array} \right.   \end{equation}

   We see from (\ref{LimitOperatorbis}) and (\ref{LimitOperatorter}) the sharp difference between the bounded selfadjoint operator $A_\var$ and the limit operator $A_0$  which is no more a bounded selfadjoint operator.  
	
	\medskip
	
  Of course,  the same remark may be made about  the homogenized problem given by (\ref{FF1}). 
	\end{Remark}
  From the technical point of view the main difficulty in the asymptotic analysis comes from the lack of compactness since we have to consider sequences of eigenvectors not bounded in $H^1(\Omega)$ so that the strong convergence in $L^2(\Omega)$ (or strong two-scale convergence in the case of homogenization) which allows to conclude the limit of an eigenvector $u_\var^k$ is still an eigenvector (i.e. $\not=0$) is not straightforward. To overcome this difficulty, we will use an extension technique (see \cite{CioPau}, \cite{Zikool}) combined with another slightly more intricate argument. 
	
	\medskip
	
  From now on and based on the previous comments, we will focus on the asymptotic analysis of the singular perturbation problem (\ref{FV2}) (the study of the reduction of dimension occurring in each cell). This kind of problems is usually encountered in the study of thin structures, see for instance \cite{Led} and \cite{Pan}. 
	
	\medskip
	
	Our main results may be stated as follows.
  
  \begin{Theorem}\label{3d-1d} 
  For each $k=1,2,...$, the sequence of eigenvalues $(\lambda_\var^k)_\var$ of (\ref{FV2}) is bounded above by the first eigenvalue $\mu_1$ of $-\Delta'$ in $H_0^1(D)$ and the associated sequence of eigenvectors $(u_\var^k)_\var$  is bounded in $L^2(0,L; H^1(C))$; 
  if for a subsequence of $\var$, $\display \lambda_\var^k \to \lambda_k $ with $ \lambda_k \not= \mu_1$, then there exists a solution  $(\lambda_k, u_0^k, v_k) \in (0, \mu_1 [ \times L^2(0,L; H^1(C)) \times H_0^1(0,L)$ of (\ref{FF2}) with $v_k \not= 0$ such that for the whole sequence $\var$, one has
   \begin{equation}\label{CV0}  
\lambda_\var^k \to \lambda_k,
  \end{equation}
   \begin{equation}\label{CV1}  
  u_\var^k \scvge u_k(y,x_3):= (\lambda_k u_0^k + 1) v_k  \, \, \hbox{strongly in} \, \,  L^2(0,L; H^1(C)), 
  \end{equation}
  
   \begin{equation}\label{CV1bis}  
  u_\var^k \chi_M \scvge  v_k \chi_M \, \, \hbox{strongly in} \, \,  L^2(C; H_0^1(0,L)).  
   \end{equation}
  Any  $\lambda_k$ such that $0 < \lambda_k < \mu_1$ is a simple eigenvalue of the limit operator $A_0$. \par
  Conversely, problem  (\ref{FF2}) admits non trivial solutions such that $0 < \lambda_k <  \mu_1$ and any $\lambda \in (0, \mu_1[ $ which is an eigenvalue of (\ref{FF2}) is a limit of a sequence $(\lambda_\var^k)_\var$ of eigenvalues of (\ref{FV2}). \par
 The unique accumulation point of the sequence $(\lambda_k)_k$ is the first eigenvalue $\mu_1$ of $-\Delta'_y$; hence $\display \lim_{k \to +\infty} \lambda_k = \mu_1$.
 \end{Theorem}
 
\begin{Remark}The property $v_k \not= 0$ may be deduced from the strong convergence (\ref{CV1}) of the eigenvectors but we prefer to write it explicitly to highlight the fact that $v_k$ is always an eigenvector of $ \display - \frac{d^2}{dx_3^2}$ with Dirichlet boundary condition.
 \end{Remark}

 Regarding the homogenization problem, the result is in all respects similar to that of $3d-1d$. We state it through the following theorem which is the homogenized version of Theorem \ref{3d-1d}. To state the results, we need the use of the two-scale convergence, see \cite{All}, \cite{Nguet}, \cite{Zhik}. We use  the notation 
  $\tscale$ (resp. $\overset{2-sc} {\longrightarrow}$) for the two-scale convergence (resp. the strong two-scale convergence).
  \medskip
  \begin{Theorem}\label{homogenization}
 For each $k=1,2,...$, the sequence of eigenvalues $(\lambda_\var^k)_\var$ of (\ref{FV1}) is bounded above by the first eigenvalue $\mu_1$ of $-\Delta'$ in $H_0^1(D)$ and the associated sequence of eigenvectors $(u_\var^k)_\var$  is bounded in $L^2(0,L; H^1(\omega))$;
 if for a subsequence of $\var$, $\display \lambda_\var^k \to \lambda_k $ with $ \lambda_k \not= \mu_1$, then there exists a solution  $(\lambda_k, u_0^k, v_k) \in (\mu_0, \mu_1[ \times L^2(0,L; H_{\#}^1(C)) \times L^2(\omega; H_0^1(0,L))$ of (\ref{FF1}) with $v_k \not= 0$ such that for the whole sequence $\var$, one has
   \begin{equation}\label{CV0bis}  
\lambda_\var^k \to \lambda_k,
  \end{equation}
 \begin{equation}\label{CV2}  
    u_\var^k \overset{2-sc}{\longrightarrow} u_k(x,y):= (\lambda_k u_0^k + 1) v_k, \end{equation} 
 with the following corrector result 
  \begin{equation}\label{Corrector}  
  \display \int_\Omega \left( \left( \left| \var \nabla'u_\var^k - \nabla'_y u_k \left(x, \frac{x'}{\var}\right) \right|^2  + \var^2 \left| \frac{\partial u_\var^k}{\partial x_3}\right|^2 \right)\chi_{F_\var}(x') + \left( \left| \nabla'u_\var^k \right|^2  + \left| \frac{\partial u_\var^k}{\partial x_3} - \frac{\partial v_k}{\partial x_3}\right|^2\right) \chi_{M_\var}(x') \right) \ dx \to 0.   \end{equation} 
   Any  $\lambda_k$ such that $0 < \lambda_k < \mu_1$ is a simple eigenvalue of the limit operator $A_0$. \par
  Conversely, any eigenvalue $\lambda \in (0, \mu_1[$ of problem (\ref{FF1}) is a limit of a sequence $(\lambda_\var^k)_\var$ of eigenvalues of (\ref{FV1}). \par
    The sequence $(\lambda_k)_k$ converges to $\mu_1$. 
\end{Theorem}

  \medskip
  \begin{Remark}
  Note that the structure of the limit spectrum is quite complicated because not only the mean value $ \display \int_{D} u^k_0 \ dy $ arising in the second equation of the limit system must be calculated by the use of the first equation of the system but the function $u^k_0$  itself depends on the corresponding eigenvalue as shown by the first equation; hence, 
  $ \display \lambda_k \Bigl( 1 + \frac{\vert D \vert }{\vert C \setminus D\vert } + \frac{ \lambda_k}{\vert C \setminus D \vert} \displaystyle \int_{D} u^k_0 \ dy \Bigr)$ which is an eigenvalue  of 
  $\display - \frac{d^2}{dx_3^2}$ is not completely known in terms of $\lambda_k$. However, we will prove (see (\ref{secondequation})) that for $ 0 < \lambda_k < \mu_1$, the second equation describing the vibrations of the string $(0,L)$ may be written as
  \begin{equation} \label{vibrationsofthestring}  
 \display - \frac{d^2 v_k}{dx_3^2} = \delta(\lambda_k) v_k \, \hbox{with}  \, \delta(\lambda):= C\lambda + C' \sum_{n=1}^\infty \frac{c^2_n \lambda^2}{\mu_n - \lambda},
 \end{equation} where $C, C'$ denote positive constants and $c_n:= \display \int_D f_n dy$ where $(f_n)_n$ denotes the orthonormal basis in $L^2(D)$ made up of eigenfunctions associated to the increasing sequence $(\mu_n)_n$ of eigenvalues of $-\Delta'_y$ with Dirichlet boundary condition. Of course, the spectrum $\sigma_0$ of the limit operator $A_0$ contains eigenvalues on the right of $\mu_1$; in particular, (\ref{vibrationsofthestring}) shows that any eigenvalue $\mu_n$ of $-\Delta'$ such that $ c_n = \display \int_D f_n dy \not=0$ is an accumulation point of $\sigma_0$. Our result states that the limits $\lambda_k$ make up a part of the spectrum $\sigma_0$ of $A_0$, namely the values of $\sigma_0$ located on the left of $\mu_1$.
 
\medskip

 Remark also that in the homogenization setting, the analogous result of the convergence (\ref{CV1bis}) is the convergence $ \display \int_\Omega \left| \frac{\partial u_\var^k}{\partial x_3} - \frac{\partial v_k}{\partial x_3}\right|^2 \chi_{M_\var}(x') \ dx \to 0$ obtained from the corrector result  (\ref{Corrector}). However, the latter does not mean that the sequence $u_\var^k \chi_{M_\var}$ converges strongly in $L^2(\omega; H_0^1(0,L))$ to $ \frac{\vert C \setminus D \vert}{\vert C \vert}  v_k= \vert C \setminus D \vert v_k$ (we have assumed $\vert C \vert = 1$) in which case this convergence would be the exact analogue of (\ref{CV1bis}). Unfortunately, because of the oscillations induced by the homogenization process, such exact analogue of (\ref{CV1bis}) is false. This is one of the few differences between the $3d-1d$ problem and the homogenization problem.
   \end{Remark}
  \medskip
    Finally we point out other possible scalings of the form $\var^\gamma \chi_{F_\var} + \var^\delta \chi_{M_\var}$ as addressed in \cite{GauSi0}, \cite{GauSi}, \cite{Sili2}. For instance in the static case, one can refer to \cite{GauSi0}. The critical case giving rise to a coupled system at the limit is the one corresponding to $ \display \lim \var^{\delta - 2} = l \in ]0, +\infty[$ which we consider here. 
		
		\medskip
		
  In order to highlight the close analogy between the $3d-1d$ limit problem and the homogenized problem, the macroscopic variable $x$ will be denoted by $x=(y,x_3), \, y \in C$ in the study of the $3d-1d$ problem for which $\Omega:=C\times (0,L)$ while in the homogenization problem $x$ will be denoted by $x=(x',x_3), \, x' \in \omega := \display \bigcup_{i \in I_\var} (\var C + \var i)$ since $ \Omega:= \display \bigcup_{i \in I_\var} (\var C + \var i) \times(0,L)$ so that each $x'\in \omega$ may be written as $x'= \var y + \var i, \, \, i \in I_\var$. In the case of a single thin structure $\Omega_\var = (\var C) \times (0,L)$,  $\Omega:= C \times (0,L)$ is obtained from  $\Omega_\var $ by the scaling $x'=\var y, \, \, y \in C$, making our notations homogeneous.
	
\medskip
  
Before proceeding to prove the results in the next sections, it should be pointed out the study can be extended to the case of operators in divergence form. In that case, we have to take into account at the limit the contribution of the anisotropy of the heavy part of the material (here the matrix) as shown in \cite{Sili}.  On the other hand, one can consider other scalings of the form $\var^\gamma \chi_{F_\var} + \var^\delta \chi_{M_\var}$ as addressed in \cite{GauSi0}, \cite{GauSi}, \cite{Sili2} in the static case. For instance in the static case and under convenient assumptions on the source term, one can consider coefficients of order $\var^\delta$ in the fiber  ${F_\var}$ and $1$ in ${M_\var}$, then loosely speaking the structure of the limit problem depends on the limit of the ratio $ \display \var^{\delta - 2}$, the critical case giving rise to a coupled system at the limit is the one corresponding to $ \display \lim \var^{\delta - 2} = l \in ]0, +\infty[$. Here we address the critical case in the framework of the Laplacian operator for the sake of simplicity and brevity. 

\medskip

  In the following we study in detail the dimension reduction problem and then indicate briefly the few technical changes needed in the proofs of the result in the framework of homogenization, see also \cite{Sili0}.

  \section{Proof of the results in the case of a single thin structure: the reduction of dimension $3d-1d$}
   \subsection{Apriori estimate on the sequence of eigenvalues and eigenvectors}
    \begin{Proposition}\label{propo1}
  For each $k=1,2,...$, the sequence $(\lambda_\var^k, u_\var^k)$ of eigenpairs of (\ref{FV2}) is bounded in $  \mathbb{R} \times L^2(0,L; H^1(C))$. There exist $(\lambda_k, u_k, v_k) \in (0, \mu_1) \times L^2(0,L; H^1(C)) \times H_0^1(0,L)$ and a subsequence of $\var$ still denoted by $\var$ such that
   \begin{equation}\label{CW1} u_\var^k \weak u_k \, \, \hbox{weakly in} \, \, L^2(0,L; H^1(C)) \, \, \hbox{and} \, \, u_k(y,x_3)= v_k(x_3) \, \, \hbox{in}\, \, M = ( C \setminus D) \times(0,L),
   \end{equation}
    \begin{equation}\label{C2} \frac{\partial u_\var^k}{\partial x_3}  \chi_M \weak \frac{dv_k}{dx_3}  \chi_M \quad \hbox{weakly in } \, \, L^2(\Omega),
   \end{equation}
   \begin{equation}\label{C3} \lambda_\var^k \to \lambda_k.
   \end{equation}
  \end{Proposition}
  \begin{proof}
  We first prove an apriori estimate on the sequence of eigenvalues which will play a key role in the sequel. Let $ \lambda^0_k $ be the  k-th eigenvalue of $ \display - \frac{d^2}{dx_3^2}$ in $(0, L)$ with homogeneous Dirichlet boundary conditions and let  $\mu_1$ be the first eigenvalue of $ - \Delta'_y$ in $ D $ with homogeneous Dirichlet boundary condition.
	
	\medskip
	
  We claim that
  \begin{equation}\label{fundamentalestimate} \forall \, \var, \quad \forall \, k=1,2,..., \quad \lambda_\var^k \leq \mu_1 + \var^2  \lambda^0_k. \end{equation}
  Indeed, we use the well known min-max formula giving  the k-th eigenvalue $\lambda_\var^k$ of (\ref{FV2}),  
  \begin{equation}\label{minmax}
\display  \lambda_\var^k =  \display \min_{V^k \subset V_s} \max_{u\in V^k} \frac{ \display \int_\Omega \left( \left(\left| \nabla_y'u \right|^2 + \var^2 \left| \frac{\partial u}{\partial x_3} \right|^2 \right) \chi_F + \left(\frac{1}{\var^2} \left| \nabla_y'u \right|^2 +  \left| \frac{\partial u}{\partial x_3} \right|^2 \right) \chi_M \right) \ dy \ dx_3}{\display \int_\Omega \left| u \right|^2 \ dy \ dx_3},
   \end{equation}
  where the space $V_s$ is defined by (\ref{l'espace Vbis}) (with $\var=1$) and the $\min$ runs over all subspaces $V^k$ of $V_s$ with finite dimension $k$. \par
  Let $\phi(y)$ be an eigenvector associated to $\mu_1$ extended by zero in $ C \setminus D$. Then $\phi(y) \psi(x_3)$ belongs to $V_s$ for any $\psi \in H_0^1(0,L)$ and $\phi \psi = 0$ in $M:= (C \setminus D) \times (0, L)$. \par
  Let $V^k$ be the subspace of $V_s$ spanned by $\left\{\phi v^1, \phi v^2, ..., \phi v^k \right\}
  $ where $v^1, v^2, ..., v^k$ denote the associated eigenvectors to the first k eigenvalues $\lambda_1^0, \lambda_2^0, ..., \lambda_k^0$  of $\display - \frac{d^2}{dx_3^2}$ with homogeneous Dirichlet boundary conditions. \par
  For any $ u= \alpha_1 \phi v^1 + ... + \alpha_k \phi v^k \, \in V^k$, we have $u=0$ in $M$ and since $v^1, v^2, ..., v^k, ..., $ is an orthonormal basis in $H_0^1(0,L)$ we also have
  \begin{equation}\label{calculminmax} \left\{\begin{array}{ll} \display  \int_\Omega u^2 dy\ dx_3 = \int_{ D} \phi^2 \ dy \ \int_0^L \bigl( \alpha_1^2 (v^1)^2+ ... + \alpha_k^2 (v^k)^2\bigr) \ dx_3 \\\\ =  \display \bigl( \alpha_1^2 + ... + \alpha_k^2\bigr) \int_{ D} \phi^2 \ dy, \\\\
 \display  \int_\Omega \vert \nabla'_y u \vert^2 dy \ dx_3 =  \bigl( \alpha_1^2 + ... + \alpha_k^2\bigr) \int_{ D} \vert \nabla'_y \phi \vert^2 \ dy \\\\
  =  \display \bigl( \alpha_1^2 + ... + \alpha_k^2\bigr) \mu_1 \int_{ D}\vert \phi \vert^2 \ dy , \\\\
 \display  \int_\Omega \var^2 \left| \frac{\partial  u}{\partial x_3}\right|^2 dy \ dx_3 =  \var^2   \int_0^L \left( \alpha_1^2 \left|\frac{dv^1}{dx_3}\right|^2+ ... + \alpha_k^2 \left|\frac{dv^k}{dx_3}\right|^2\right) \ dx_3  \int_{ D}\vert  \phi \vert^2 \ dy, \\\ 
    \display =  \var^2   \bigl( \alpha_1^2 \lambda_1^0 + ... + \alpha_k^2 \lambda_k^0\bigr) \int_{D} \vert \phi \vert^2 \ dy \leq \var^2  \lambda_k^0  \bigl( \alpha_1^2 + ... + \alpha_k^2\bigr) \int_{D}\vert \phi \vert^2 \ dy.
   \end{array} \right.
  \end{equation}
  Note that the equality occurring in the fifth line of (\ref{calculminmax}) is a consequence of  the equation $ - \Delta'_y \phi = \mu_1 \phi \quad \hbox{in} \, \,  D$. Hence, using (\ref{calculminmax}) in the min-max formula above, we get estimate (\ref{fundamentalestimate}). 
	
	\medskip
	
  We obtain that $\lambda_k \in (0, \mu_1)$ by passing to the limit (for a subsequence of $\var$) in (\ref{fundamentalestimate}). 
	
	\medskip
  
  Turning back to (\ref{FV2}) and taking $u_\var^k$ (with $\parallel u_\var^k \parallel_{L^2(\Omega)} =1$) as a test function, we get 
  \begin{equation}\label{aprioriestimteontheeigenvectors}
  \display \int_\Omega \left( \left(\left| \nabla' u_\var^k \right|^2 + \var^2 \left| \frac{\partial u_\var^k}{\partial x_3}\right|^2\right)\chi_F + \left(\frac{1}{\var^2}\left| \nabla' u_\var^k \right|^2 +  \left| \frac{\partial u_\var^k}{\partial x_3}\right|^2\right)\chi_M \right) \ dy \ dx_3 = \lambda_\var^k \leq K.
  \end{equation}
  The last estimate implies that $\nabla'u_\var^k$ is bounded in $L^2(\Omega)$ and thus $u_\var^k$ is bounded in $L^2(0,L; H^1(C))$. Hence, there exist a sequence of $\var$ and $u_k \in L^2(C; H_0^1(0,L))$ such that the convergence (\ref{CW1}) holds true. 
	
	\medskip
	
  One has  $\nabla'u_\var^k\chi_M(y) \weak \nabla'u_k \chi_M$ weakly in $L^2(\Omega)$. But $ \nabla'u_\var^k \chi_M $ which is bounded in $L^2(\Omega)$ by $C\var$ strongly converges to zero in $L^2(\Omega)$. Hence, $\nabla'u_k \chi_M = 0$ which means that $u_k = v_k(x_3)$ for some $v_k \in L^2(0,L)$ a.e. in $M$.
  The sequence  $u_\var^k\chi_M(y)$ (note that the characteristic functions $\chi_F$ and $\chi_M$ depend only on the horizontal variable $y$) is bounded in $L^2(C; H_0^1(0,L))$ since $\display \frac{\partial u_\var^k}{\partial x_3} \chi_M $ is bounded in $L^2(\Omega)$ so that for a subsequence $ \display  \frac{\partial u_\var^k}{\partial x_3} \chi_M \weak \, \frac{\partial u_k}{\partial x_3} \chi_M = \frac{ dv_k}{d x_3} \chi_M \quad \hbox{weakly in} \, \, L^2(\Omega) $. Hence $v_k \in H_0^1(0, L)$ and the convergence (\ref{C2}) holds true. The proof of  Proposition \ref{propo1} is complete.
    \end{proof}
    
   \subsection{The limit problem associated to (\ref{FV2})}
   Choosing a test function in (\ref{FV2}) in the form  $\phi = \bar u  $ with $ \bar u = \bar v(x_3) \, \hbox{in} \, \, M$ and $(\bar u, \bar v) \in V_s \times H_0^1(0,L)$, we get from (\ref{FV2}) 
    \begin{equation} \label{FV2bis}    \displaystyle \int_\Omega \left( \left(\nabla' u_\var^k \nabla' \bar u + \var^2 \frac{\partial u_\var^k}{\partial x_3} \frac{\partial \bar u}{\partial x_3}\right) \chi_{F} +    \frac{\partial u_\var^k}{\partial x_3} \frac{d \bar v}{dx_3} \chi_{M} \right) dy \ dx_3   = \lambda_\var^k \display \int_\Omega u_\var^k \bar u\ dy \ dx_3.   
\end{equation}
   Passing to the limit in this equation, we get with the help of (\ref{propo1})
    \begin{equation} \label{FV2ter} \left\{\begin{array}{ll}  (u_k,v_k) \in L^2(0,L; H^1(C)) \times H_0^1(0,L),  \, \,  u_k =  v_k \,\, \hbox{in} \,\, M, \\\\  \displaystyle \int_\Omega \left( \nabla' u_k \nabla' \bar u \chi_{F} +   \frac{ dv_k}{dx_3} \frac{d\bar v}{dx_3} \chi_{M} \right) dy \ dx_3   = \lambda_k \display \int_\Omega u_k \bar u\ dy \ dx_3, \\\\ \forall \, (\bar u,\bar v) \in V_s \times H_0^1(0,L),  \, \, \bar u = \bar v \, \, \hbox{in}  \, M.   
  \end{array} \right. 
	\end{equation}
  Finally a density argument allows to extend  (\ref{FV2ter}) to all test functions $\bar u \in L^2(0,L; H^1(C))$ such that $\bar u = \bar v \, \hbox{in} \, M$ and $ \bar v \in H_0^1(0,L)$. 
	
	\medskip
	
  Choosing successively in (\ref{FV2ter}) $ \bar u  \in  \, L^2(0,L; H^1(C))$ such that 
  $ \bar u = 0 \, \hbox{in} \, \, M$ and then $ \bar u  \in  \, L^2(0,L; H^1(C))$ such that $ \bar u = \bar v \, \in H_0^1(0,L) \, \hbox{ almost everywhere in} \, \, \Omega$ and bearing in mind the geometry of $\Omega:= C\times (0,L) = \bigl( (C\setminus D) \cup D \bigr) \times (0,L)$, we get that the limit problem (\ref{FV2bis}) may be split into two equations leading to the following equivalent system   
  \begin{equation} \label{FF2bis} \left\{\begin{array}{ll}  u_k(y,x_3) \in L^2((0,L); H^1(C)),\quad
 \displaystyle - \Delta'_y u_k (y,x_3) = \lambda_k u_k \,\, \, \,  \hbox{in} \,\,  D \times (0, L),  \\\\
  u_k = v_k \quad   \hbox{on} \, \, \, \partial D \times (0, L), 
   \\\\
 v_k \in  H^1_0(0, L)),  \quad \display - \frac{d^2 v_k}{dx_3^2} = \lambda_k v_k  +\frac{ \lambda_k}{\vert C \setminus D \vert} \displaystyle \int_{D} u_k \ dy \quad \hbox{in} \, \, \,  (0,L).
  \end{array} \right. \end{equation}

  \medskip
  \begin{Remark}\label{u_ketv_k} Eigenvectors of (\ref{FF2bis}) corresponding to eigenvalues $\lambda_k < \mu_1$ are pairs $(u_k,v_k)$ made up of two inseparable elements. In particular, if $v_k = 0$ then $u_k = 0$ as shown by (\ref{FF2bis}). Indeed, otherwise $u_k$ should be an eigenvector of $-\Delta'_y$ associated to the eigenvalue $ \lambda_k < \mu_1$ which is a contradiction. Conversely if $u_k = 0$ then $v_k = 0$ since almost everywhere in $(0,L)$, we have $v_k = u_k$ on the boundary of $D$. Hence, the eigenvectors $(u_k,v_k)$ of the limit operator are such that $u_k \not=0$ and $v_k \not=0$.
  \end{Remark}
  We now prove that (\ref{FF2}) and (\ref{FF2bis}) are equivalent if one defines $u_k$ by (\ref{expressiondeu_k}) and then we will improve the lower bound of the limit eigenvalues using (\ref{FF2}).
  \begin{Proposition}\label{equivalence}
  If $(\lambda_k, u_k, v_k)$ solves the system (\ref{FF2bis}) with $ 0 < \lambda_k < \mu_1$, then $v_k\not=0$ and $u_k$ writes as \begin{equation} \label{expressiondeu_k} u_k (y,x_3) = ( \lambda_k  u_0^k(y) + 1) v_k(x_3) \end{equation}
  where  $(\lambda_k, u_0^k, v_k)$ solves  (\ref{FF2}). Furthermore, there exists a positive constant $\mu_0$ depending both on $\mu_1$ and on the first eigenvalue of $ \display -\frac{d^2}{dx_3^2}$ in $H_0^1(0,L)$ such that $\lambda_k \geq \mu_0$ for all $k$. 
  \end{Proposition}
  \begin{proof}
  Assume that $(u_k, v_k)$ is a non trivial solution of (\ref{FF2bis}), i.e, $(u_k, v_k)$ is an eigenvector of the limit operator. Then according to the Remark \ref{u_ketv_k} above, $v_k\not=0$ and $u_k \not=0$.  
	
	\medskip
	
  Dividing by $v_k$ in the first system of (\ref{FF2bis}), one can check easily that $ w_k:= \frac{u_k}{v_k} - 1$ is the unique solution of 
 \begin{equation}\label{afterdividing} \left\{\begin{array}{ll}  \displaystyle - \Delta'_y w_k = \lambda_k w_k + \lambda_k  \,\, \, \,  \hbox{in} \,\,  D \\\\
     w_k = 0 \quad   \hbox{on} \, \, \, \partial D.    \end{array} \right. \end{equation}
Note that the uniqueness of $w_k$ is ensured since $ \lambda_k < \mu_1$ belongs to the resolvent of $-\Delta'_y$. On the other hand, the function $ \lambda_k u_0^k$ where $u_0^k$ is defined in (\ref{FF2}) is also a solution of (\ref{afterdividing}) so that the equality $ w_k:= \frac{u_k}{v_k} - 1 = \lambda_k u_0^k$ holds true and therefore (\ref{expressiondeu_k}) follows. Using (\ref{expressiondeu_k}) in (\ref{FF2bis}) we get (\ref{FF2}). 

\medskip

We now make more precise the lower bound of the sequence of eigenvalues and we prove at the meanwhile that $\display   \int_{D}  u^k_0(y) \ dy > 0.$ 

\medskip

Multiplying the first equation of (\ref{FF2}) by $u^k_0$ and using $\mu_1^{-1}$ as the constant (it is in fact the best one) in the Poincar\'e's inequality, we get
\begin{equation}\label{Lowerboundforeigenvalues}
\display  \int_D u^k_0(y) \ dy =  \int_D \vert \nabla'_y u^k_0(y)\vert^2 \ dy - \lambda_k \int_D \vert u^k_0(y)\vert^2 \ dy \geq \left(1- \frac{\lambda_k}{\mu_1}\right) \int_D \vert \nabla'_y u^k_0(y)\vert^2 \ dy.
\end{equation}
On the other hand, the first eigenvalue $\mu_1$ is characterized by 
\begin{equation} \display \mu_1 = \inf_{ u \in  H_0^1(D) \setminus \lbrace 0 \rbrace} \frac{\parallel \nabla'_y u \parallel^2_{L^2(D)}}{\parallel  u \parallel^2_{L^2(D)}}.\end{equation}
Hence the following estimate holds true
\begin{equation} \label{lowerboundbis}
 \display \int_{D} \vert \nabla'_y u^k_0(y)\vert^2 \ dy \geq \mu_1 \int_{ D} \vert u^k_0(y)\vert^2 \ dy.
\end{equation}
From (\ref{Lowerboundforeigenvalues}), we derive with the help of (\ref{lowerboundbis})
\begin{equation} \label{lowerboundter}
 \display ( \mu_1 - \lambda_k)  \int_{ D} \vert  u^k_0(y)\vert^2 \ dy \leq \int_{ D}  u^k_0(y) \ dy \leq \sqrt{\vert  D\vert} \left(\int_{ D} \vert  u^k_0(y)\vert^2 \ dy \right)^\frac{1}{2},
\end{equation}
and then from (\ref{lowerboundter}) we deduce
\begin{equation} \label{lowerboundlast}
 \display  0 <  \int_{ D}  u^k_0(y) \ dy \leq  \frac{\vert  D\vert} {\mu_1 - \lambda_k}.
\end{equation}
By virtue of the last equation in (\ref{FF2}), $ \display  \hat \lambda_k:= \lambda_k \bigl( 1 + \frac{\vert D \vert }{\vert C \setminus D \vert }+ \frac{ \lambda_k}{\vert  C \setminus D \vert} \displaystyle \int_{D} u^k_0 \ dy \bigr) $ is an eigenvalue of $\display - \frac{d^2}{dx_3^2}$ so that $ \hat \lambda_k \geq \lambda_0$ where $\lambda_0$ denotes the first eigenvalue of $\display - \frac{d^2}{dx_3^2}$. Using the second inequality of (\ref{lowerboundlast}) we get 
\begin{equation} \label{lowerboundlastbis}
 \display  \lambda_k \left( 1 + \frac{\vert  D\vert }{\vert C \setminus D \vert } +   \lambda_k \frac{\vert  D\vert} {\vert C \setminus D\vert (\mu_1 - \lambda_k)}\right) \geq \hat \lambda_k \geq \lambda_0.
\end{equation}
Hence, $\display \lambda_k \geq \mu_0:=  \phi^{-1}(\lambda_0)$ where  $\phi$ is the continuous increasing function defined on $ (0, \mu_1) $ by $$\phi(t) = t \left( 1 +\frac{\vert D\vert }{\vert C \setminus D \vert } +   t \frac{\vert  D\vert} {\vert C\setminus D \vert (\mu_1 - t)}\right).$$
   \end{proof} 
   
   So far, we have not yet proved that $(u_k, v_k)$ is indeed an eigenvector of the limit operator; this is the purpose of the next subsection.
   \subsection{The strong convergence of the eigenvectors}
  We prove the following compactness result.
  \begin{Proposition}\label{strongconvergence}
  For each $k$, there exists a subsequence of $\var$ such that the sequence of solutions $u_\var^k$ of (\ref{FV2}) converges strongly in $L^2(\Omega)$ to the eigenvector $u_k$ of (\ref{FF2bis}). 
  \end{Proposition}
  \begin{proof}
  One can extend $u_\var^k$ from $M$ to the whole $\Omega$ in such a way the extension $ {U_\var^k } $ fulfills $ U_\var^k \in V_s, \, \, \,  U_\var^k  = u_\var^k \, \, \hbox{in} \, \, M$ and  \begin{equation} \label{normeduprolongement} \display  \parallel \nabla'U_\var^k \parallel_{L^2(\Omega)} \leq K \parallel \nabla'u_\var^k \parallel_{L^2(M)}, \, \, \, \,  \left\| \frac{\partial U_\var^k}{\partial x_3}  \right\|_{L^2(\Omega)} \leq K \left\| \frac{\partial u_\var^k}{\partial x_3}   \right\|_{L^2(M)}. \end{equation}
  Note that the extension only affects the horizontal variable $y$  so that the Dirichlet boundary condition on the upper and lower faces of $\Omega$ ($x_3 = 0 $ or $x_3  = L$ ) is preserved, see for instance \cite{Bre}, \cite{CioPau}, \cite{Zikool}. \par
  In addition, one can assume that such extension satisfies the following equation
  \begin{equation} \label{equationontheextension}  
  \display - \Delta'_y U_\var^k - \var^2 \frac{\partial^2 U_\var^k}{\partial x_3^2} = 0 \quad  \hbox{in} \, \, F.
    \end{equation}
  
  Indeed, if (\ref{equationontheextension}) is not true for $U_\var^k$, then one can introduce the function $W_\var^k$ as the unique solution of 
  \begin{equation}\label{equationontheextensionbisFV} \left\{\begin{array}{ll}  \displaystyle W_\var^k \in V, \\\\
   \display \int_F \left( \frac{1}{\var^2} \nabla'_y W_\var^k \nabla'_y \phi +  \frac{\partial W_\var^k}{\partial x_3} \frac{\partial \phi}{\partial x_3} \right) dy \ dx_3 =   \int_F \left( \frac{1}{\var^2} \nabla'_y U_\var^k \nabla'_y \phi +  \frac{\partial U_\var^k}{\partial x_3} \frac{\partial \phi}{\partial x_3} \right) dy \ dx_3,  \\\\ \forall \, \phi \in V, 
   \end{array} \right.
 \end{equation}
  where $V:=  \big\lbrace u \in V_s, u = 0 \, \, \, \hbox{on} \, \, \, \partial D \times (0,L) \big\rbrace $ (recall that $V_s:= V^\var_s$ with $\var = 1$ where $ V^\var_s$ is defined by (\ref{l'espace Vbis})).
  Hence, $V$ is the subspace of $V_s$ of functions vanishing in $M$. By the Lax-Milgram Theorem we get existence and uniqueness for $W_\var^k$. Choosing $\phi \in C_0^\infty(F)$, the last equation leads to 
   \begin{equation}\label{equationontheextensionbis}  \displaystyle - \frac{1}{\var^2} \Delta'_y W_\var^k -  \frac{\partial^2 W_\var^k}{\partial x_3^2} = - \frac{1}{\var^2} \Delta'_y U_\var^k -  \frac{\partial^2 U_\var^k}{\partial x_3^2} \quad \hbox{in} \, \, \, F.
\end{equation}
 On the other hand, using equation (\ref{equationontheextensionbisFV})  with $\phi = W_\var^k$, we get the following estimate with the help of (\ref{normeduprolongement}) and (\ref{aprioriestimteontheeigenvectors})
   \begin{equation}\label{estimateonW} \left\{\begin{array}{ll}  \displaystyle
   \display  \left\| \frac{1}{\var}\nabla'W_\var^k \right\|_{L^2(F)} +   \left\|  \frac{\partial W_\var^k}{\partial x_3}   \right\|_{L^2(F)} \leq K \left(
\left\| \frac{1}{\var}\nabla'U_\var^k \right\|_{L^2(F)} +  \left\| \frac{\partial U_\var^k}{\partial x_3}  \right\|_{L^2(F)} \right)  \leq \\\\
\leq  K \display  \left(
\left\| \frac{1}{\var}\nabla'u_\var^k \right\|_{L^2(M)} +  \left\| \frac{\partial u_\var^k}{\partial x_3}  \right\|_{L^2(M)} \right)  \leq K. \end{array} \right.
   \end{equation}
  Multiplying equation (\ref{equationontheextensionbis}) by $\var^2$, we see that $ \tilde u_\var^k$ defined by $ \tilde u_\var^k = U_\var^k - W_\var^k$ is indeed an extension which fulfills equation (\ref{equationontheextension}) and preserves the apriori estimate (\ref{normeduprolongement}). Note that functions of $V$ may be extended by zero inside $M$ so that $\tilde u_\var^k$ is still an extension of $ u_\var^k$ from $M$ to the whole $\Omega$. 
	
	\medskip
	
  In the sequel, we will still denote the extension of $u_\var^k$ satisfying (\ref{normeduprolongement}) and (\ref{equationontheextension})  by $U_\var^k$. 
	
	\medskip
	
   Consider now the sequence defined in $\Omega $ by $z_\var^k = u_\var^k - U_\var^k$. If we prove that $z_\var^k$ admits  a strongly converging subsequence in $L^2(\Omega)$ then we can deduce the existence of such subsequence for $u_\var^k$ since $U_\var^k$ is bounded in $H^1(\Omega)$ by virtue of  (\ref{normeduprolongement}) and (\ref{aprioriestimteontheeigenvectors}) and therefore admits a strongly converging subsequence in $L^2(\Omega)$ according to the Rellich imbedding Theorem.
	
	\medskip
	
  We first derive the following equation on $z_\var^k$ by the use of (\ref{newstrongformulationS'}) together with (\ref{equationontheextension})
  \begin{equation}\label{equationsurzepsilon} \left\{\begin{array}{ll}  \displaystyle z_\var^k \in V_s, \quad
   -  \Delta'_y z_\var^k -  \var^2\frac{\partial^2 z_\var^k}{\partial x_3^2} = \lambda_\var^k z_\var^k + \lambda_\var^k U_\var^k    \quad \hbox{in} \, \, \, F,\\\\
  z_\var^k = 0 \quad \hbox{on} \, \, \, \partial D \times (0,L).
   \end{array} \right. \end{equation}
   
  Since $u_\var^k$ and $U_\var^k$ are bounded respectively in $L^2(0,L; H^1(C))$ and $H^1(\Omega)$, the sequence $z_\var^k$ is bounded in $L^2(0,L; H^1(C))$. Hence, there exist a subsequence and $z_k \in L^2(0,L; H^1(C))$ such that $$ \display z_\var^k \weak z_k \, \, \, \hbox{weakly in} \, \,  L^2(0,L; H^1(C)).$$ Therefore, denoting by $U_k$ the weak limit in $H^1(\Omega)$ of the corresponding subsequence $U_\var^k$, one can pass easily to the limit in (\ref{equationsurzepsilon}) to get the equation
  \begin{equation}\label{equationsurz_k} \left\{\begin{array}{ll}  \displaystyle z_k \in  L^2(0,L; H^1(C)), \quad 
   -  \Delta'_y z_k  = \lambda_k z_k + \lambda_k U_k    \quad \hbox{in} \, \, \, F,\\\\
  z_k = 0 \quad \hbox{on} \, \, \, \partial D \times (0,L).
   \end{array} \right. \end{equation}
  Note that by construction, $z_\var^k = 0 $ in $M= (C \setminus \overline D) \times (0,L)$ so that the  convergence $$ z_\var^k \chi_M(y)  \weak \, z_k \chi_M(y) \, \, \hbox{weakly in } \, \, L^2(\Omega)$$ shows that $ z_k = 0 \,  \,  \,  \hbox{in} \, \, M$ which is equivalently expressed by the boundary condition of (\ref{equationsurz_k}). \par
  More generally, given a bounded sequence $(f_\var)$ in $L^2(\Omega)$ and $f \in L^2(\Omega)$, we now consider equations of the form 
  \begin{equation}\label{equationsgenerales} \left\{\begin{array}{ll}  \displaystyle w_\var \in V_s, \quad
   -  \Delta'_y w_\var -  \var^2\frac{\partial^2 w_\var}{\partial x_3^2} = \lambda_\var^k w_\var + f_\var    \quad \hbox{in} \, \, \, F,\\\\
  w_\var = 0 \quad \hbox{on} \, \, \, \partial D \times (0,L),
   \end{array} \right. \end{equation}
   and 
   \begin{equation}\label{equationsurw} \left\{\begin{array}{ll}  \displaystyle w \in  L^2(0,L; H^1(C)), \quad
   -  \Delta'_y w  = \lambda_k w + f   \quad \hbox{in} \, \, \, F,\\\\
  w = 0 \quad \hbox{on} \, \, \, \partial D \times (0,L).
   \end{array} \right. \end{equation}
   Regarding the sequence  of solutions  of (\ref{equationsgenerales}), the following lemma holds true.
   \begin{Lemma} Assume that $\lambda_\var^k \to \lambda_k$ with $ \lambda_k < \mu_1$ and that 
   $f_\var \weak f \, \, \, \hbox{weakly in } \, \, L^2(\Omega)$. Then the sequence $w_\var$ is bounded in $L^2(0,L; H^1(C))$ and for the whole sequence $\var$,   $w_\var \weak w \, \, \, \hbox{weakly in} \, \, L^2(0,L; H^1(C))$ where $w$ is the unique solution of (\ref{equationsurw}).
   \end{Lemma}
   \begin{proof}
   First, one can check that $\mu_1$ satisfies the inequality $\mu_1 \leq \delta_1$ where $\delta_1$ denotes the first eigenvalue of the operator $ \display  -  \Delta'_y -  \var^2\frac{\partial^2 }{\partial x_3^2} $ in $H_0^1(F)$ so that for sufficiently small $\var$, $\lambda_\var^k$ belongs to the resolvent of that operator and consequently the existence and the uniqueness of $w_\var$ follow. Hence, we only have to prove that $w_\var$ is bounded in $L^2(0,L; H^1(C))$, the limit problem (\ref{equationsurw}) satisfied by $w$  can be established exactly by the same process already used in the proof of (\ref{equationsurz_k}).
	
	\medskip
	
   The main ingredient to get that apriori estimate relies on the Poincar\'e inequality    \begin{equation}\label{Poincaré}
   \display \int_{ D} \vert u \vert^2 \  dy \leq \frac{1}{\mu_1} \int_{D} \vert \nabla'_y u \vert^2 \ dy \quad \forall \, u \in H_0^1(D),
   \end{equation}
   combined with the assumption $ \lambda_k < \mu_1 $. 
	
	\medskip
	
Multiplying equation (\ref{equationsgenerales}) by $w_\var$ and integrating, we get
   \begin{equation}\label{w_varbis} \int_0^L \int_{D} \vert \nabla'w_\var \vert^2 \ dy dx_3 \leq \lambda_\var^k  \int_0^L \int_{ D} \vert w_\var \vert^2 \ dy dx_3 + 
   \parallel f_\var \parallel_{L^2(\Omega)} \parallel w_\var \parallel_{L^2(F)}.
   \end{equation} 
   Choosing  $ u = w_\var(., x_3)$ with $x_3 \in (0, L)$ and integrating  (\ref{Poincaré}) over $(0,L)$,  we infer
  \begin{equation}\label{Poincarésurw_var} \int_0^L \int_{ D} \vert w_\var \vert^2 \ dy dx_3 \leq \frac{1}{\mu_1}\int_0^L \int_{ D} \vert \nabla'_y w_\var \vert^2 \ dy dx_3.
  \end{equation}
 Let $\delta > 0$ be such that $ 0 < \lambda_k < \delta < \mu_1$. Turning back to (\ref{w_varbis}) and using (\ref{Poincarésurw_var}), we get for $\var$ sufficiently small, 
 \begin{equation}\label{encorew_var}
\display \left(1 - \frac{\delta}{\mu_1}\right) \int_0^L \int_{ D} \vert \nabla'w_\var \vert^2 \ dy dx_3 \leq   
   \parallel f_\var \parallel_{L^2(\Omega)} \parallel w_\var \parallel_{L^2(F)}.
\end{equation}
Since $f_\var$ is bounded in $L^2(\Omega)$, applying once again inequality (\ref{Poincarésurw_var}), we derive from (\ref{encorew_var}) the estimate
\begin{equation}\label{encorew_varbis}
\display \int_0^L \int_{D} \vert \nabla'w_\var \vert^2 \ dy dx_3 \leq  K. 
  \end{equation}
The estimates (\ref{Poincarésurw_var}) and (\ref{encorew_varbis}) show that $w_\var$ is bounded in $L^2(0,L; H^1(D))$ and thus in $L^2(0,L; H^1(C))$ since $ w_\var$ is equal to zero in $ C \setminus D$.
   \end{proof}
   We continue the proof of the Proposition \ref{strongconvergence} in the following way.\par
  \noindent Multiplying  equations (\ref{equationsurzepsilon}) and (\ref{equationsgenerales})
   respectively by $w_\var$ and by $z_\var^k$ and integrating we get
   \begin{equation}\label{zepsilonetw_var}
   \left\{\begin{array}{ll}  \displaystyle \int_F \left( \nabla'z_\var^k \nabla'w_\var + \var^2 \frac{\partial z_\var^k}{\partial x_3}  \frac{\partial w_\var}{\partial x_3} \right) \ dy dx_3 = \lambda_\var^k \int_F z_\var^k w_\var \ dy dx_3 + \lambda_\var^k \int_F U_\var^k w_\var \ dy dx_3 = \\\\
    \display  \lambda_\var^k \int_F w_\var z_\var^k  \ dy dx_3 + \int_F f_\var z_\var^k \ dy dx_3.  \end{array} \right. 
\end{equation}
Since $U_\var^k$ is bounded in $H^1(\Omega)$, there exist a subsequence of $\var$ and $U_k \in H^1(\Omega)$ such that $ U_\var^k \weak U_k$ weakly in $ H^1(\Omega)$ and strongly in $L^2(\Omega)$ by virtue of the Rellich embedding Theorem. Therefore for that a subsequence, we get from (\ref{zepsilonetw_var}) with the help of Lemma 2.6
\begin{equation} \label{1erelimite}
\display \lim_{\var \rightarrow 0} \int_F f_\var z_\var^k \ dy dx_3 = \display \lim_{\var \rightarrow 0} \lambda_\var^k \int_F U_\var^k w_\var \ dy dx_3 =  \lambda_k \int_F U_k w  \ dy dx_3. \end{equation}
On the other hand, one can multiply (\ref{equationsurz_k}) and (\ref{equationsurw})
  respectively by $w$ and by $z_k$ and integrate to obtain
 \begin{equation} \label{2iemelimite}
 \left\{\begin{array}{ll}  \displaystyle \display \int_F \nabla' z_k \nabla' w \ dy dx_3 = \int_F \nabla' w \nabla' z_k \ dy dx_3 = \lambda_k \int_F z_k w \ dy dx_3  + \lambda_k \int_F U_k w \ dy dx_3  \\\\
 \display = \lambda_k \int_F w z_k  \ dy dx_3  +  \int_F f z_k \ dy dx_3. \end{array} \right.
  \end{equation}
  Combining (\ref{1erelimite}) and (\ref{2iemelimite}), we get
  \begin{equation} \label{presque}
  \display \lim_{\var \rightarrow 0} \int_F f_\var z_\var^k \ dy dx_3 = \lambda_k \int_F U_k w \ dy dx_3 = \int_F f z_k \ dy dx_3.
  \end{equation}
  Choosing in particular $ f_\var = z_\var^k $ which converges weakly in $L^2(\Omega)$ to $f=z_k$, we obtain 
  \begin{equation}\label{convergencedelanorme}
  \lim_{\var \rightarrow 0} \int_F (z_\var^k)^2 \ dy dx_3 = \int_F (z_k)^2 \ dydx_3,
  \end{equation}
  which implies the strong convergence of the subsequence $z_\var^k$ and therefore the strong convergence of the corresponding subsequence of $u_\var^k$. Hence Proposition \ref{strongconvergence} is proved.
 \end{proof}
 We now proceed to complete the proof of Theorem \ref{3d-1d}.
 \subsection{Proof of Theorem \ref{3d-1d}}
 The strong convergence in $L^2(\Omega)$ of the eigenvectors when $\lambda_k < \mu_1$ is proved in Proposition 2.4. We use it to prove the convergence of the sequence of energies  from which we obtain immediately (\ref{CV1}) and (\ref{CV1bis}). \par
  Consider the sequence
 \begin{equation}\label{CVFENERGIE}
 \display J_\var = \int_\Omega \left( \left(\left| \nabla'u_\var^k - \nabla'u_k \right|^2 + \var^2 \left| \frac{\partial u_\var^k}{\partial x_3}\right|^2 \right) \chi_F + \left( \frac{1}{\var^2}\left| \nabla'u_\var^k \right|^2 +   \left| \frac{\partial u_\var^k}{\partial x_3} - \frac{dv_k}{ dx_3}\right|^2 \right) \chi_M \right) \ dy dx_3.
 \end{equation}
 Choosing $u_\var^k$ and $(u_k,v_k)$ as test functions respectively in (\ref{FV2}) and in (\ref{FV2ter}), we get with the help of the weak convergences proved in Proposition 2.1 and of the strong convergence proved in Proposition 2.4,
 \begin{equation}\label{CVFENERGIE2}
 \left\{\begin{array}{ll}   \display J_\var = \lambda_\var^k \int_\Omega  (\vert u_\var^k \vert^2 dy dx_3 + \lambda_k \int_\Omega  \vert u_k \vert^2 dy dx_3  - 2 \int_\Omega \bigl( \nabla'u_\var^k \nabla'u_k \chi_F +   \frac{\partial u_\var^k}{\partial x_3} \frac{dv_k}{ dx_3}\chi_M \bigr) \ dy dx_3 \\\\
  \display  \longrightarrow  \, \, \,  2 \lambda_k\int_\Omega \vert u_k \vert^2 dy dx_3 - 2\lambda_k\int_\Omega \vert u_k \vert^2 dy dx_3 = 0. \end{array} \right.
 \end{equation}
 Hence the weak convergences stated in Proposition 2.1 are in fact strong convergences; in particular, keeping in mind Proposition 2.4, we get the strong convergences stated in Theorem \ref{3d-1d}.  \par  
 We have proved above that  $\lambda_k$ is an eigenvalue of the limit problem (in the sense of (\ref{LimitOperatorbis})) if and only if  $\lambda_k$ satisfies (\ref{FF2}). In the sequel, a number $\lambda$ satisfying (\ref{FF2}) will be called an eigenvalue of the limit problem (\ref{FF2}). \par
  We now prove there exist non trivial solutions for the system (\ref{FF2}) and that any $\lambda \in (\mu_0, \mu_1)$ which satisfies  (\ref{FF2}) may be attained as a limit of a sequence $(\lambda_\var^k)_\var$; by this we can conclude that (\ref{LimitOperatorbis}) has no other eigenvalues on the left of $\mu_1$ than those obtained from the limits of the eigenvalues $\lambda_\var^k$ and thus we can list all its eigenvalues in increasing order. It is then clear that for a fixed $k$, we cannot have two subsequences $\var$ and $\var'$ with two different limits for $\lambda_\var^k$ and $\lambda^k_{\var'}$ since this would lead to add a new element to the set of eigenvalues of (\ref{FF2}); hence for each $k$, (\ref{CV0}) holds for the whole sequence $\var$.
 \par
 To prove the existence of non trivial solutions $\begin{pmatrix} u_0^k \\ v_k \end{pmatrix}$ for the system (\ref{FF2}) with $\lambda_k < \mu_1$ leading to non trivial solutions $\begin{pmatrix} u_k \\ v_k \end{pmatrix}$ for  (\ref{LimitOperatorbis})) where $u_k:= (\lambda_k u_0^k + 1) v_k$,  it is sufficient to show that one can find solutions $\begin{pmatrix} u_0^k \\ v_k \end{pmatrix}$ of (\ref{FF2}) with $v_k \not=0$. \par
\noindent  $u_0^k$ is uniquely determined by the first equation of (\ref{FF2}) since $  \lambda_k < \mu_1$ and if $(f_n)_n$ is the orthonormal basis in $L^2(D)$ made up of eigenfunctions  associated to the increasing sequence $(\mu_n)_n$ of eigenvalues of $-\Delta'_y$, one can get from the first equation of (\ref{FF2})
\begin{equation} \label{expressionintermsofeigenfunctions} u_0^k = \sum_{n=1}^\infty \frac{c_n f_n}{\mu_n - \lambda_k}; \, \hbox{where} \,  c_n= \int_D f_n \ dy. \end{equation}
Replacing the mean value of $u_0^k$ in the second equation of (\ref{FF2}), we derive
\begin{equation} \label{secondequation}  
 \display - \frac{d^2 v_k}{dx_3^2} = \delta(\lambda_k) v_k; \, \hbox{with} \, \delta(\lambda):= C\lambda + C' \sum_{n=1}^\infty \frac{c^2_n \lambda^2}{\mu_n - \lambda},
 \end{equation} where $C, C'$ denote positive constants.\par \noindent
 
Let $(\gamma_j, v_j)$ be an eigenelement of $\display -\frac{d^2}{dx_3^2}$ in $H_0^1(0,L)$. Since $  \delta $ is a strictly positive increasing function over $(0, \mu_1) $, there exists $ \lambda_{k_j} \in (0, \mu_1)$ such that  $\gamma_j= \delta(\lambda_{k_j})$, so that the second equation of  (\ref{FF2}) may be written as  $ \display - \frac{d^2 v_j}{dx_3^2} = \delta(\lambda_{k_j}) v_j$, taking $ v_{{k_j}}:=v_j$. Hence for $\lambda_k < \mu_1$, the pair $(u_0^k, v_{{k_j}}) $ is a non trivial solution for any $j=1, 2,...$
 
 We now argue by contradiction to prove that any $\lambda \in (\mu_0, \mu_1[$ which is an eigenvalue of (\ref{FF2}) may be attained as a limit of a sequence $(\lambda_\var^k)_\var$ for some $k$. 

\medskip

 If for any $k$ and for any sequence $\var$,  $\lambda_\var^k$ does not converge to $\lambda$, then there exists a neighborhood of $\lambda$ which does not contain any $\lambda_\var^k$ for all $k$. In other words, $\lambda$ belongs to the resolvent of the operator $A_\var$ defined by (\ref{newstrongformulationS'}) and there exists $\kappa > 0$ such that  $ \display \mbox{dist}(\lambda, \sigma(A_\var)) \geq \kappa$. Hence, for any $ f \in L^2(\Omega)$ there exists $ u_\var \in \, D(A_\var)$
  such that 
  \begin{equation} \label{lambda_kdans la résolvante} \display A_\var u_\var = \lambda u_\var + f \quad \hbox{in} \, \, \Omega.\end{equation}
In addition,
\begin{equation} \label{normedelaresolvante} \display \parallel (A_\var - \lambda)^{-1} \parallel =  \frac{1}{\mbox{dist}(\lambda, \sigma(A_\var))} \leq \frac{1}{\kappa}.\end{equation}
Inequality (\ref{normedelaresolvante}) provides the following estimate
\begin{equation} \label{suiteborneeL^2} \display \parallel u_\var \parallel_{L^2(\Omega)}= \parallel (A_\var - \lambda)^{-1}(f) \parallel_{L^2(\Omega)}  \leq \frac{1}{\kappa}\parallel f \parallel_{L^2(\Omega)}.\end{equation}
Hence, the sequence $u_\var$ is bounded in $L^2(\Omega)$.
Multiplying (\ref{lambda_kdans la résolvante}) by $\phi \in V_s$ and integrating we get
\begin{equation} \label{lambda_kdans la résolvante2}  \left\{\begin{array}{ll} \display
\int_\Omega \left( \left(\nabla' u_\var \nabla' \phi + \var^2 \frac{\partial u_\var}{\partial x_3} \frac{\partial \phi}{\partial x_3}\right) \chi_{F} +   \left(\frac{1}{\var^2}\nabla' u_\var \nabla' \phi +  \frac{\partial u_\var}{\partial x_3} \frac{\partial \phi}{\partial x_3}\right) \chi_{M} \right) dy \ dx_3   = \\\\
\lambda \display \int_\Omega u_\var \phi \ dy \ dx_3 +  \int_\Omega f \phi \ dy \ dx_3, \quad \forall \, \phi \in V_s.   \end{array} \right.
  \end{equation}
 
 Taking $\phi = u_\var$ in (\ref{lambda_kdans la résolvante2}), we get the same apriori estimates as those obtained for the sequence $u_\var^k$ and therefore one can pass to the limit  $\var \to 0$ in (\ref{lambda_kdans la résolvante2}) to derive the equation   
 \begin{equation} \label{FFavectermesource} \left\{\begin{array}{ll}  u(y,x_3) \in L^2((0,L); H^1(C)),\quad
 \displaystyle - \Delta'_y u (y,x_3) = \lambda u + f  \,\, \, \,  \hbox{in} \,\,  D \times (0, L),  \\\\
  u = v \quad   \hbox{on} \, \, \, \partial D \times (0, L), 
   \\\\
 
 v \in  H^1_0(0, L),  \quad \display - \frac{d^2 v}{dx_3^2} = \lambda v +\frac{ \lambda}{\vert C \setminus  D \vert} \displaystyle \int_{D} u\ dy +  \frac{1}{\vert C \setminus  D \vert} \displaystyle \int_{C } f \ dy \quad \hbox{in} \, \, \,  (0,L).
  \end{array} \right. \end{equation}
Choosing $f(y,x_3) =  g(x_3) \chi_{C \setminus D} (y)$  (which implies $ f = 0$ in $ D$)  with an arbitrary $g\in L^2(0,L)$, the second equation in (\ref{FFavectermesource}) reduces to
 \begin{equation} \label{vdanslarésolvante}
 v \in  H^1_0(0, L),  \quad \display - \frac{d^2 v}{dx_3^2} =  \lambda v +\frac{ \lambda}{\vert C \setminus D \vert} \displaystyle \int_{ D} u\ dy +  g  \quad \hbox{in} \, \, \,  (0,L).
 \end{equation}
Note that $v \not= 0$ for $g \not=0$. Indeed if $v=0$,  the first equation in (\ref{FFavectermesource}) would imply $u=0$ since we have chosen $f$ such that $f=0$ in $D$ and $ \lambda < \mu_1$ is not an eigenvalue of $-\Delta'_y$. Therefore equation  (\ref{vdanslarésolvante}) would give $g=0$ which is a contradiction. \par
Therefore, one can express $u$ as $u = (\lambda u_0 + 1) v$ where the pair $ (\lambda, u_0)$ solves the first equation of (\ref{FF2}). so that (\ref{vdanslarésolvante}) takes the form
\begin{equation} \label{vdanslarésolvantebis}
 v \in  H^1_0(0, L)),  \quad \display - \frac{d^2 v}{dx_3^2} =  \lambda \left( 1 + \frac{ \vert D \vert }{\vert C\setminus D \vert}  +  \frac{ \lambda} {\vert C \setminus D \vert} \displaystyle \int_{D} u_0\ dy \right) v +  g  \quad \hbox{in} \, \, \,  (0,L).
 \end{equation}
 On the other hand, by hypothesis, $\lambda$ is an eigenvalue of (\ref{FF2}) so that the last equation of (\ref{FF2}) with the same $u_0$ as in (\ref{vdanslarésolvantebis}) shows that $\display  \lambda \left( 1 + \frac{\vert D \vert}{ \vert C\setminus D\vert }  +  \frac{ \lambda} {\vert C \setminus D \vert} \displaystyle \int_{ D} u_0\ dy \right) $ is an eigenvalue of $\display -\frac{d^2}{dx_3^2}$. This is a contradiction since
equation (\ref{vdanslarésolvantebis}) valid for all $g \in L^2(0,L)$ means that the number $\display  \lambda \left( 1 + \frac{\vert D \vert}{ \vert C\setminus D\vert }  +  \frac{ \lambda} {\vert C \setminus D \vert} \displaystyle \int_{ D} u_0\ dy \right) $ belongs to the resolvent of $\display -\frac{d^2}{dx_3^2}$. 

\medskip

 We prove now that $\display \lim_{k \to +\infty} \lambda_k = \mu_1$. 

\medskip

 Since $\lambda_k \in (\mu_0, \mu_1)$ for any $k$, the sequence $(\lambda_k)_k$ admits at least an accumulation point and each accumulation point $\lambda$ is such that $\mu_0 \leq \lambda \leq \mu_1$. Assume that there exists an accumulation point $\lambda$ such that $ \lambda < \mu_1$. There exists a subsequence $(\lambda_{k_n}, u_0^{k_n}, v_{k_n})$ of solutions of (\ref{FF2}) such that $ \display \lim_{n \to + \infty} \lambda_{k_n} = \lambda$. Hence the following equation takes place for all $n$
 \begin{equation} \label{accumulationpoint}
 - \Delta' u_0^{k_n} = \lambda_{k_n} u_0^{k_n}  + 1 \quad \hbox{in} \, \, D.
  \end{equation}
  Let $ \delta$ be a positive number such that $ \lambda < \delta < \mu_1$. For $n$ large enough we have $\lambda_{k_n} \leq \delta$ so that applying the Poincar\'e inequality
  \begin{equation}\label{encorePoincare} 
\int_{ D} \lvert  u \rvert^2 \ dy  \leq \frac{1}{\mu_1} \int_{ D} \lvert \nabla'_y u \rvert^2 \ dy  \quad \forall \, u \in H_0^1(D),
 \end{equation} 
 after multiplying (\ref{accumulationpoint}) by $u_0^{k_n}$, we get for $n$ large enough
 \begin{equation}\label{encorePoincarebis} 
\int_{D} \lvert \nabla'_y u_0^{k_n} \rvert^2 \ dy  \leq \frac{\delta}{\mu_1} \int_{  D} \lvert \nabla'_y u_0^{k_n} \rvert^2 \ dy  +  \int_{ D} \lvert  u_0^{k_n} \rvert \ dy. 
 \end{equation} 
 Applying successively the Cauchy-Schwarz inequality and (\ref{encorePoincare}) in the last integral of (\ref{encorePoincarebis}), we infer
 \begin{equation}\label{encorePoincareter} 
\Bigl( 1 - \frac{\delta}{\mu_1}\Bigr) \int_{ D} \lvert \nabla'_y u_0^{k_n} \rvert^2 \ dy  \leq \sqrt{\lvert  D \rvert }    \sqrt{\frac{1}{\mu_1}} \sqrt{\int_{  D} \lvert \nabla'_y u_0^{k_n} \rvert^2 \ dy}.
 \end{equation} 
 Therefore, $(u_0^{k_n})_n$ is bounded in $H_0^1 (D)$ and one can assume (possibly for another subsequence) that  $(u_0^{k_n})_n$ converges weakly to $u_0$  in $ H_0^1(D)$. In particular we have that $\display  \lim_{n\to +\infty} \int_{ D} u_0^{k_n} \ dy =  \int_{ D} u_0 \ dy$. On the other hand $(\lambda_{k_n}, u_0^{k_n}, v_{k_n})$ being a solution of (\ref{FF2}), the following equation (recall that $v_{k_n} \not=0$ ) 
 \begin{equation}
 \display - \frac{d^2 v_{k_n}}{dx_3^2} =  \lambda_{k_n} \left( 1 + \frac{ \vert  D \vert }{\vert C \setminus D \vert}  +  \frac{ \lambda_{k_n}} {\vert C \setminus D \vert} \displaystyle \int_{ D} u_0^{k_n} \ dy \right) v_{k_n} \quad \forall \, n,
 \end{equation}
 shows that the number $\mu$ defined by $\mu:= \display \lambda \Bigl( 1 + \frac{\vert D \vert}{ \vert C\setminus D\vert }  +  \frac{ \lambda} {\vert C \setminus D \vert} \displaystyle \int_{D} u_0 \ dy \Bigr)$ is a finite accumulation point  of the spectrum of $\display - \frac{d^2}{dx_3^2}$ since  $\mu=  \display \lim_{n\to + \infty} \mu_n$ where $\mu_n:= \display \lambda_{k_n} \Bigl( 1 + \frac{\vert D \vert}{ \vert C \setminus  D\vert } +  \frac{ \lambda_{k_n}} {\vert C \setminus D \vert} \displaystyle \int_{ D} u_0^{k_n} \ dy \Bigr)$. This is a contradiction since it is well known that such spectrum is in fact an increasing sequence which tends to $+\infty$.  \par
 The last point which remains to prove is that all the limiting eigenvalues are simple and  that  $u_\var^k$ converges to $u_k$ for the whole sequence $\var$. Assuming that $\lambda_k$ is a simple eigenvalue, the proof of the convergence of the eigenvectors for the whole sequence $\var$  is known since the work of \cite{Van} (see also \cite{CioPau}). We sketch it in the vectorial setting for the convenience of the reader. \par
 Assume that $\begin{pmatrix} u_k \\ v_k \end{pmatrix}$ is an eigenvector associated to the simple eigenvalue $\lambda_k$. 
 Using the fact that the eigenvalues converge for the whole sequence $\var$, it is easy to check  that the multiplicity of $\lambda_k$ is  equal or greater than that of $\lambda_\var^k$; hence $\lambda_\var^k$ is simple and there are only two eigenvectors  satisfying $\display \int_\Omega \lvert u_\var^k \rvert^2 \ dx = 1$, namely $u_\var^k$ and $- u_\var^k$.  Among these two eigenvectors, we choose the one satisfying the inequality 
 \begin{equation}\label{orientation}
  \display \int_\Omega \left(u_\var^k \chi_F u_k + u_\var^k \chi_M v_k \right)  \ dy dx_3  >  0.
  \end{equation}
  Therefore if $\var'$ is a subsequence such that $\begin{pmatrix} u_{\var'}^k \chi_F \\ u_{\var'}^k \chi_M \end{pmatrix}$ strongly converges in $(L^2(\Omega))^2$ to the eigenvector
  $\begin{pmatrix} \hat u \chi_F \\ \hat v \chi_M \end{pmatrix}$ associated to $\lambda_k$ , we get by passing to the limit in (\ref{orientation}),
  \begin{equation}\label{orientationbis}
  \display \int_\Omega (\hat u \chi_F u_k + \hat v \chi_M v_k)  \ dy dx_3  >  0.  \end{equation}
   On the other hand, $\begin{pmatrix} u_k \chi_F \\ v_k \chi_M \end{pmatrix} = \begin{pmatrix} \hat u_k \chi_F \\ \hat v_k \chi_M \end{pmatrix}$ or $\begin{pmatrix} u_k \chi_F \\ v_k \chi_M \end{pmatrix} = - \begin{pmatrix} \hat u_k \chi_F \\ \hat v_k \chi_M \end{pmatrix}$ since $\lambda_k$ is a simple eigenvalue. The last equality is excluded thanks to (\ref{orientationbis}) so that any subsequence is such that $\begin{pmatrix} u_{\var'}^k \chi_F \\ u_{\var'}^k \chi_M \end{pmatrix}$ strongly converges in $(L^2(\Omega))^2$  to $\begin{pmatrix} u_k \chi_F \\ v_k \chi_M \end{pmatrix}$. 
	
	\medskip
	
   Let us now prove that all the limit eigenvalues are simple eigenvalues. 
	
	\medskip
	
	Assume that for some $k$,  (\ref{LimitOperatorbis}) holds true for two orthogonal eigenvectors $ \display 
 \begin{pmatrix} u_k \\ v_k \end{pmatrix} $  and  $ \display 
 \begin{pmatrix} \bar u_k \\ \bar v_k \end{pmatrix} $ in $\display L^2(D) \times L^2(0,L)$.  By assumption, we have
 \begin{equation} \label{orthogonality} 
\display \int_0^L \int_{D} u_k \bar u_k  dy dx_3 +  \vert C \setminus D\vert  \int_0^L  v_k \bar v_k  dx_3  = 0.
 \end{equation}
 We know that  $u_k$ and $\bar u_k$ are given respectively by $u_k(y,x_3) = ( \lambda_k u_0^k(y) +1)v_k(x_3)$ and 
$ \bar u_k(y,x_3) = ( \lambda_k u_0^k(y)+1)\bar v_k(x_3)$ where $ u_0^k(y)$ given by the first equation of (\ref{FF2}) depends only on the eigenvalue $\lambda_k$. 

\medskip

Turning back to (\ref{orthogonality}), we infer
\begin{equation} \label{orthogonalitybis} 
\display  \int_0^L \left( \left(\int_{D} \left( \lambda_k u_0^k(y) + 1\right)dy\right)^2 + \vert C \setminus D\vert \right) v_k(x_3) \bar v_k(x_3) dx_3 = 0.
 \end{equation}
 As remarked above $v_k$ and $\bar v_k$ are always eigenvectors of the operator $\display - \frac{d^2}{dx_3^2}$ with Dirichlet condition so that (\ref{orthogonalitybis}) and the second equation of (\ref{FF2})  would mean that $ v_k$ and $\bar v_k$ eigenvectors associated to the eigenvalue $ \display  \lambda_k \Bigl( 1 + \frac{\vert D \vert }{\vert C \setminus D \vert }+ \frac{ \lambda_k}{\vert C \setminus D \vert} \displaystyle \int_{ D} u^k_0 \ dy \Bigr)$ are orthogonal in $L^2(0,L)$. This is a contradiction since all the eigenvalues of  $\display - \frac{d^2}{dx_3^2}$ with Dirichlet condition are simple eigenvalues.

\medskip

 The proof of Theorem \ref{3d-1d}  is now complete. 

\medskip
 
 Finally, let us indicate briefly in the following short section how to derive the analogous theorem in the homogenization setting using the same approach as in the reduction of dimension.

\section{Proof of Theorem \ref{homogenization} }
In the spirit of the above section, the natural idea is to choose a test function vanishing outside the set $F_\var$ of fibers to get the apriori estimate on the sequence of eigenvalues. To that aim, we consider an eigenvector $\phi(y)$ corresponding to the first eigenvalue of $- \Delta'_y$ in $H_0^1(D)$. We extend $\phi$ by zero over $C \setminus D$ and then by periodicity to the whole $\mathbb R^2$. The  k-th eigenvalue $\lambda_\var^k$ of (\ref{FV1}) is given by the same min-max formula, namely
  \begin{equation}\label{minmax2}
\display  \lambda_\var^k =  \display \min_{V^k \subset V_h} \max_{u\in V^k} \frac{\display \int_\Omega \left(  \var^2\vert \nabla u \vert^2  \chi_{F_\var} + \vert \nabla u \vert^2  \chi_{M_\var} \right) \ dx' \ dx_3}{\display \int_\Omega \vert u \vert^2 \ dx' \ dx_3}.
   \end{equation}
   For each $\var$, we choose $V_\var^k \subset V_h$ as the subspace spanned by  $\left\{\phi(\frac{x'}{\var}) v^1, \phi(\frac{x'}{\var}) v^2, ..., \phi(\frac{x'}{\var}) v^k \right\} $ with the same $v^1, v^2, ..., v^k$ as those defined in the previous section, i.e., $k$ normalized orthogonal eigenvectors associated to the first k eigenvalues of $\display -\frac{d^2}{dx_3^2}$ in $H_0^1(0,L)$.   \par
Hence, by construction, the functions of $V_\var^k$ vanish in $M_\var$ so that making the change of variable $ x':= \var y + \var i, \,  y \in D$ in each cell, we can perform the same calculations as those of (\ref{calculminmax}) to get for $u \in V_\var^k$, 
  \begin{equation}\label{calculminmaxbis} \left\{\begin{array}{ll} \display  \int_\Omega u^2 dx' \ dx_3 = \sum_{i \in I_\var} \int_{\var  D + \var i} \phi^2 \left(\frac{x'}{\var}\right) \ dx' \ \int_0^L \bigl( \alpha_1^2 (v^1)^2 + ... + \alpha_k^2 (v^k)^2\bigr) \ dx_3 \\\
   = \display \bigl( \alpha_1^2 + ... + \alpha_k^2\bigr) \var^2 \sum_{i \in I_\var} \int_{D} \phi^2(y) \ dy, \\\\
 \display  \int_\Omega  \var^2 \vert \nabla'_{x'} u \vert^2 dx' \ dx_3 =  \bigl( \alpha_1^2 + ... + \alpha_k^2\bigr) \sum_{i \in I_\var} \var^2 \int_{\var  D + \var i} \left| \nabla'_{x'} \phi \left(\frac{x'}{\var}\right) \right|^2 \ dx' = \\\\
   \display \bigl( \alpha_1^2 + ... + \alpha_k^2\bigr) \var^4 \sum_{i \in I_\var} \int_{ D} \frac{1}{\var^2} \vert \nabla'_{y} \phi (y) \vert^2 dy
  \display = \bigl( \alpha_1^2 + ... + \alpha_k^2\bigr) \var^2 \mu_1 \sum_{i \in I_\var} \int_{ D}\vert \phi(y) \vert^2 dy, \\\\
 \display  \int_\Omega \var^2 \left| \frac{\partial  u}{\partial x_3}\right|^2 dx' \ dx_3 =  \var^2   \int_0^L \left( \alpha_1^2 \left(\frac{dv^1}{dx_3}\right)^2+ ... + \alpha_k^2 \left(\frac{dv^k}{dx_3}\right)^2\right) \ dx_3 \   
\var^2  \sum_{i \in I_\var} \int_{ D}\vert  \phi(y) \vert^2 dy \\\ 
   \qquad  \display =  \var^4   \bigl( \alpha_1^2 \lambda_1^0 + ... + \alpha_k^2 \lambda_k^0\bigr) \sum_{i \in I_\var} \int_{ D}\vert \phi \vert^2 \ dy \leq \var^4  \lambda_k^0  \bigl( \alpha_1^2 + ... + \alpha_k^2\bigr) \sum_{i \in I_\var} \int_{ D}\vert \phi \vert^2 dy,
   \end{array} \right.
  \end{equation} in such a way the following estimate holds true
   \begin{equation}\label{calculminmaxlast}  \display 
   \lambda_\var^k \leq \frac{ \display  \Bigl(\mu_1  + \var^2 \lambda_k^0\Bigr)  \Bigl( \alpha_1^2 + ... + \alpha_k^2\Bigr) \var^2 \sum_{i \in I_\var} \int_{ D}\vert \phi \vert^2 \ dy}{ \display \var^2\Bigl( \alpha_1^2 + ... + \alpha_k^2\Bigr)  \sum_{i \in I_\var} \int_{ D} \phi^2(y) \ dy } = \display \mu_1 + \var^2 \lambda_k^0,
   \end{equation} 
   which is exactly the same estimate as that obtained in (\ref{fundamentalestimate}). \par
   \begin{Remark}\label{Comparaison}
   It is interesting to note in the proof of (\ref{calculminmaxlast}), we have chosen a test function verifying the same properties as those of the $3d-1d$ case, namely: null in the matrix and with the regularity $ H_0^1(0,L)$ for almost all $x'$.
   \end{Remark}
    Remark \ref{Comparaison} is of general relevance since the other proofs in the homogenization setting 
    are similar in all points to the corresponding ones in the $3d-1d$ problem, the main reason being that the vertical variable is not concerned by the homogenization process which occurs only with respect to the horizontal variable $x'$ in such a way basically, the local $3d-1d$ effect is repeated periodically in the horizontal plane. 
    Hence all the proofs take up exactly the 3d-1d case while sticking to two principles: Dirichlet condition on $x_3=0$ or $x_3=L$ both for the $3d-1d$ problem and the homogenization problem and when $x_3$ plays the role of a parameter as it is the case for example in equation (\ref{equationsurz_k}), it is x that will play the role of a parameter in the homogenization problem. Indeed for instance, the natural formulation of equation (\ref{equationsurzepsilon}) in the homogenization setting is the following one
     \begin{equation}\label{equationsurzepsilonhomog.} \left\{\begin{array}{ll}  \displaystyle z_\var^k \in V_h, \quad
   -  \Delta'_{x'} z_\var^k -  \var^2\frac{\partial^2 z_\var^k}{\partial x_3^2} = \lambda_\var^k z_\var^k + \lambda_\var^k U_\var^k    \quad \hbox{in} \, \, \, F_\var,\\\\
  z_\var^k = 0 \quad \hbox{on} \, \, \, \partial D_\var^{i} \times (0,L),
   \end{array} \right. \end{equation}
   in such a way passing to the two-scale limit in (\ref{equationsurzepsilonhomog.}), we get the equivalent of (\ref{equationsurz_k})
  \begin{equation}\label{equationsurz_khomog.} \left\{\begin{array}{ll}  \displaystyle z_k \in  L^2(\Omega; H_{\#}^1(C)), \quad 
   -  \Delta'_y z_k  = \lambda_k z_k + \lambda_k U_k    \quad \hbox{in} \, \, \, \Omega \times D,\\\\
  z_k = 0 \quad \hbox{on} \, \, \, \Omega \times \partial D.
   \end{array} \right. \end{equation}
   The same approach may be applied to the other proofs following exactly the same steps and replacing the weak (resp. strong) convergence in $L^2(\Omega)$ by the two-scale (resp. strong two-scale) convergence.

     \end{document}